\documentclass{article}
\usepackage{authblk} 

\usepackage[T1]{fontenc}
\usepackage[utf8]{inputenc}
\usepackage[english]{babel}
\usepackage{enumerate,amsmath,amsthm}
\usepackage{amssymb}
\usepackage{bookman}
\usepackage{enumitem,etoolbox,hyperref}
\usepackage{eulervm} 
\usepackage{tabu}
\usepackage{multirow}
\usepackage{graphicx} 
\usepackage{wrapfig} 
\usepackage[font=small,labelfont=bf]{caption} 
\graphicspath{ {pictures/} } 
\usepackage{placeins} 
\allowdisplaybreaks 
\usepackage[normalem]{ulem} 
\usepackage{tikz} 

\newcommand{\stkout}[1]{\ifmmode\text{\sout{\ensuremath{#1}}}\else\sout{#1}\fi}
\def\dout{\bgroup
 \markoverwith{\lower-0.2ex\hbox
 {\kern-.03em\vbox{\hrule width.2em\kern0.45ex\hrule}\kern-.03em}}%
 \ULon}
\MakeRobust\dout

\newcommand{\ru}[1]{\rule{0pt}{#1 em}}
\newcommand{\tupsingle}{\ru{3}} 

\input{mycommands.sty}

\input{matrixsymb.sty}

\newlength\titlebox 

\theoremstyle{definition}
\newtheorem{definition}{Definition}

\theoremstyle{plain}
\newtheorem{theorem}{Theorem}
\newtheorem{lemma}[theorem]{Lemma}

\newtheorem{proposition}[theorem]{Proposition}
\newtheorem{corollary}{Corollary}[theorem]

\theoremstyle{remark}
\newtheorem{remark}{Remark}
\newtheorem{example}{Example}

\newcommand{\btrans}[1]{\mathbb{B}\left[#1\right]}

\input{amsbibstyle.sty}
\title{On the fourth moment of a random determinant}
\author[ ]{Dominik Beck}
\affil[ ]{\small Faculty of Mathematics and Physics, Charles University, Prague}
\affil[ ]{\textit {\href{mailto:beckd@karlin.mff.cuni.cz}{beckd@karlin.mff.cuni.cz}}}
\date{\today}

\begin{document}
\maketitle
\begin{abstract}
In this paper, we generalise the formula for the fourth moment of a random determinant to account for entries with asymmetric distribution. We also derive the second moment of a random Gram determinant.
\end{abstract}

\tableofcontents

\section{Introduction}
Let $X_{ij}$'s be i.i.d. (independent and identically distributed) random variables and $A = (X_{ij})_{n \times n}$ a random square matrix having these variables as its entries. We are interested in expressing the moments of the determinant $|A|$, that is
\begin{equation}
    f_k (n) = \Exx |A|^k,
\end{equation}
as a function of moments $m_r = \mathbb{E} {X_{ij}^r}$. It is easy to see that $f_k(n)$ is a polynomial in $m_1,m_2,\ldots,m_k$. Moreover, when $k$ is odd (and $n>1$), then $f_k(n)$ is automatically zero due to the anti-symmetric property of a determinant. The only nontrivial cases are hence when $k$ is even. An equivalent formulation of the problem is to find the generating function
\begin{equation}
    F_k(t) = \sum_{n=0}^\infty\frac{t^n}{(n!)^2} f_k(n),
\end{equation}
from which one could deduce $f_k(n)$ by its Taylor expansion. However, this definition of the generating function makes sense only for $k \leq 5$, otherwise it does not in general define an analytic function of $t$ on an interval containing zero. Although, treated formally, it satisfies \cite{prekopa1967random}
\begin{equation}
\frac{\partial F_k(t)}{\partial \mu_k} = t F_k(t). 
\end{equation}
Sometimes, we restrict the distribution of $X_{ij}$'s:

\begin{itemize}
    \item We say $X_{ij}$'s follow a \textbf{symmetric} distribution, if the odd moments are  equal to zero up to the order $k$ (that is, $m_{2l+1}=0$ for $2l+1\leq k$). We denote $f^{\mathrm{sym}}_k(n)$ and $F^{\mathrm{sym}}_k(t)$ the corresponding $k$-th moment of the random determinant formed by those random variables, and its generating function, respectively.
    \item We say $X_{ij}$'s follow a \textbf{centered} distribution (or equivalently, we say $X_{ij}$'s are centered random variables) if $m_1 = 0$. For those variables, we consider $f^{\mathrm{cen}}_k(n)$ and $F^{\mathrm{cen}}_k(t)$ in the same way.
\end{itemize}

The problem of finding the moments of a random determinant was studied extensively in a series of papers published in the 1950s \cite{fortet1951random}\cite{forsythe1952extent}\cite{nyquist1954distribution}\cite{turan1955problem}. For $k=2$, there is a well known general formula \begin{equation}
    f_2(n) = n! (m_2 + m_1^2(n-1)) (m_2 - m_1^2)^{n-1}
\end{equation}
attributed originally to Fortet \cite{fortet1951random} as a special case of a more general setting, although the formula itself could be derived in a much more elementary way \cite{stanley_fomin_1999}.

However, no such formula was available for higher moments given $X_{ij}$'s being generally distributed, although there are three notable special cases:
\begin{enumerate}
\item Nyquist, Rice and Riordan \cite{nyquist1954distribution} derived
\begin{equation}
    F_4^{\mathrm{sym}}(t) = \frac{e^{t(m_4-3 m_2^2)}}{(1-m_2^2 t)^3},
\end{equation}
from which they obtained
\begin{equation}
    f^{\mathrm{sym}}_4(n) = (n!)^2 m_2^{2n} \sum_{j=0}^n \frac{1}{j!} \left(\frac{m_4}{m_2^2} - 3\right)^j \binom{n-j+2}{2}.
\end{equation}
In fact, this formula holds even if $X_{ij}$'s follow just a centered distribution. That is, $ f^{\mathrm{cen}}_4(n) =  f^{\mathrm{sym}}_4(n)$. This is due to the fact that $m_3$ appears always as a product $m_1 m_3$ in the $f_4(n)$ polynomial.

\item In the same paper, they also derived that if $X_{ij}$'s follow a standard normal distribution, then $f_k(n)$ could be expressed for \emph{any} even $k=2m$ as
\begin{equation}
    f_{2m}(n) = (n!)^m\prod_{r=0}^{m-1} \binom{n+2r}{2r}.
\end{equation}
A more elementary derivation of this result was later given by Pr\'{e}kopa \cite{prekopa1967random}.

\item Just recently, Lv and Potechin \cite{LP2022} also obtained an explicit formula for $f^{\mathrm{sym}}_6(n)$. After some simplifications, their result is equivalent to
\begin{equation}
f^{\mathrm{sym}}_6(n)=(n!)^2 m_2^{3n}\sum_{j=0}^n \sum_{i=0}^j\frac{(1\!+\!i) (2\!+\!i) (4\!+\!i)! }{48 (n-j)!}\binom{14\!+\!j\!+\!2i}{j-i} \left(\frac{m_6}{m_2^3}-15\frac{m_4}{m_2^2}+30\right)^{n-j} \left(\frac{m_4}{m_2^2}-3\right)^{j-i}.
\end{equation}
However, due to nontrivial $m_3^2$ terms, $f_6^{\mathrm{sym}}(n)$ and $f_6^{\mathrm{cen}}(n)$ do not generally coincide. Luckily, using the same methods as in their paper, it can be easily derived that
\begin{equation}
f^{\mathrm{cen}}_6(n) = (n!)^2 m_2^{3n} \sum _{j=0}^n \sum _{i=0}^j \sum_{k=0}^{n-j} \frac{(1\!+\!i) (2\!+\!i) (4\!+\!i)! }{48 (n-j-k)!}\binom{10}{k}\binom{14\!+\!j\!+\!2i}{j-i} n_6^{n-j-k} n_4^{j-i}n_3^k,
\end{equation}
where
\begin{equation}
n_6 = \frac{m_6}{m_2^3}-10 \frac{m_3^2}{m_2^3}-15 \frac{m_4}{m_2^2} + 30, \qquad\qquad n_4 = \frac{m_4}{m_2^2}-3, \qquad\qquad n_3 = \frac{m_3^2}{m_2^3}.
\end{equation}
\end{enumerate}

More generally, denote $U = (X_{ij})_{nxp}$ a rectangular matrix with i.i.d. random variable entries $X_{ij}$. This time, we are interested in expressing the moments of the determinant $|U^T U|$ as a polynomial of the moments $m_r = \Exx X_{ij}^r$ of the entries. For even $k$, we denote
\begin{equation}\label{Fkto}
    f_{k}(n,p) = \Exx |U^T U|^{k/2} \qquad \text{and} \qquad     F_k(t, \omega) = \sum_{n=0}^\infty \sum_{p=0}^n\frac{(n-p)! }{n!p!} t^p \omega^{n-p} f_k(n,p)
\end{equation}
with $f_k(n,0)=1$ by definition (we put $|U^TU| = 1$ when $p=0$). Notice that, when $n=p$, we get by the multiplicative property of determinant
\begin{equation}\label{f0F0}
    f_k(n,n) = f_k(n) \qquad \text{and thus} \qquad F_k(t,0) = F_k(t).
\end{equation}
Again, restricting the distribution of $X_{ij}$'s, we write
\begin{itemize}
    \item \textit{(centered distribution)} $f_k(n,p) = f^{\mathrm{cen}}_k(n,p)$ and $F_k(t,\omega) = F^{\mathrm{cen}}_k(t,\omega)$ if $m_1 = 0$; and similarly
    \item \textit{(symmetrical distribution)} $f_k(n,p) = f^{\mathrm{sym}}_k(n,p)$ and $F_k(t,\omega) = F^{\mathrm{sym}}_k(t,\omega)$ if $m_1 = m_3 = m_5 = \ldots = 0$.
\end{itemize}
The fact that $f_k(n,p)$ is a polynomial in $m_p$ leads to the important equality
\begin{equation}
    f^{\mathrm{cen}}_k(n,p) = f^{\mathrm{sym}}_k(n,p) \qquad \text{valid for} \qquad k=2,4.
\end{equation}
When $k\geq 6$, $f^{\mathrm{cen}}_k(n,p)$ contains extra products of even powers of odd moments ($m_3^2,\ldots$). Let us have a quick overview of some special cases. As a simple consequence of \textbf{Cauchy-Binet formula} \cite{stanleyAA2015}, we have
\begin{equation}
    f_2(n,p) = p! \binom{n}{p} (m_2 + m_1^2(p-1)) (m_2 - m_1^2)^{p-1}
\end{equation}
It turns out, this was the only known general formula. Although, as a special case, Dembo \cite{dembo1989random} showed
\begin{enumerate}
\item For any even $k$, formally,
\begin{equation}
\frac{\partial F_k(t,\omega)}{\partial \mu_k} = t \, F_k(t,\omega). \end{equation}
\item If $X_{ij}$'s follow the standard normal distribution, then ($k=2m$ even)
\begin{equation}\label{GenPrekopa}
        f_{2m}(n,p) = p!^m\prod_{r=0}^{m-1} \binom{n+2r}{n-p+2r}.
    \end{equation}
    However, this result is older and there had even been a generalization of it based on known properties of the non-central Wishart distribution: Let $X_{ij} \sim N(\mu,\sigma^2)$, then (see Theorem 10.3.7 in \cite{muirhead1982aspects})
\begin{equation}\label{ThmGauss}
    f_{2m}(n,p) = p!^m \sigma ^{2 m p} \left(\prod_{r=0}^{m-1} \binom{n+2r}{n-p+2r}\right) \sum_{s=0}^m \binom{m}{s} \frac{(n-2)!!}{(n+2s-2)!!} \left(\frac{n p
   \mu ^2}{\sigma ^2}\right)^s.
\end{equation}
 \item For symmetrical distribution of $X_{ij}$'s,
    \begin{equation}
        F^{\mathrm{sym}}_4(t,\omega) = \frac{e^{t(m_4 - 3m_2^2)}}{(1-m_2^2 t)^2(1-\omega-m_2^2 t)}, \qquad  f^{\mathrm{sym}}_4(n,p) = p!^2 \binom{n}{p} m_2^{2p} \sum_{j=0}^p \frac{1}{j!} \left(\frac{m_4}{m_2^2} - 3\right)^j \binom{n-j+2}{n-p+2}.
    \end{equation}
    Note that, letting $\omega=0$ (or $p=n$), we recover the formulae of Nyquist, Rice and Riordan \cite{nyquist1954distribution}
\begin{equation}\label{NRR}
    F_4^{\mathrm{sym}}(t) = \frac{e^{t(m_4 - 3 m_2^2)}}{(1-m_2^2t)^3}, \qquad\qquad     f_4^{\mathrm{sym}}(n) =(n!)^2 m_2^{2n} \sum_{j=0}^n \frac{1}{j!} \left(\frac{m_4}{m_2^2} - 3\right)^j \binom{n-j+2}{2}.
\end{equation}
\end{enumerate}

\subsection*{Main results}
The aim of this paper is to generalize the result of Nyquist, Rice and Riordan \cite{nyquist1954distribution} to express the full $f_4(n)$. That is, with $X_{ij}$'s being generally distributed. Furthemore, we aim to generalize the result of Dembo \cite{dembo1989random} to express the full $f_4(n,p)$. We present the following theorems and their corollaries:
\begin{theorem}\label{MainThm}
\begin{equation}
    F_4(t) = \frac{e^{t(\mu_4 - 3 \mu_2^2)}}{(1-\mu_2^2t)^5}\left(1+\sum_{k=1}^6 p_k t^k\right),
\end{equation}
where
\begin{equation*}
\begin{split}
p_1 & = m_1^4+6 m_1^2 \mu_2-2 \mu_2^2+4 m_1 \mu_3, \qquad p_2 = 7 m_1^4 \mu_2^2-6 m_1^2 \mu_2^3+\mu_2^4+12 m_1^3 \mu_2 \mu_3-8 m_1 \mu_2^2 \mu_3+6 m_1^2 \mu_3^2,\\[-0.1em]
p_3 & = 2 m_1 \left(2 m_1^3 \mu_2^4-6 m_1^2 \mu_2^3 \mu_3+2 \mu_2^4 \mu_3+3 m_1^3 \mu_2 \mu_3^2-6 m_1 \mu_2^2 \mu_3^2+2 m_1^2
   \mu_3^3\right),\\[-0.1em]
p_4 & = m_1^2 \mu_3^2 \left(m_1^2 \mu_3^2-6 m_1^2 \mu_2^3+6 \mu_2^4-8 m_1 \mu_2^2 \mu_3\right), \qquad p_5 = 2 m_1^3 \mu_2^2 \mu_3^3 \left(2 \mu_2^2-m_1 \mu_3\right), \qquad p_6 = m_1^4 \mu_2^4 \mu_3^4
\end{split}
\end{equation*}
and (central moments)
\begin{equation*}
    \mu_2 = m_2-m_1^2, \qquad
    \mu_3 = m_3-3 m_1 m_2+2 m_1^3, \qquad
    \mu_4 = m_4-4m_1m_3+6 m_1^2 m_2-3 m_1^4.
\end{equation*}
\end{theorem}
\begin{corollary} Defining $\mu_j$ as above, we have, by Taylor expansion,
\begin{equation}
    f_4(n) =(n!)^2 \mu_2^{2n} \sum_{j=0}^n \frac{1}{j!} \left(\frac{\mu_4}{\mu_2^2} - 3\right)^j \sum_{i=-2}^4 q_i \binom{n-j+i}{i},
\end{equation}
where
\begin{align*}
    & q_{-2} = \frac{m_1^4 m_3^4}{\mu_2^8}, \qquad
    q_{-1} = -\frac{4 m_1^3 \mu_3^3 \left(\mu_2^2+m_1 \mu_3\right)}{\mu_2^8}, \qquad
    q_0 = \frac{6 m_1^2 \mu_3^2 \left(\mu_2^4+2 m_1 \mu_2^2 \mu_3+m_1^2 \mu_3^2-m_1^2 \mu_2^3\right)}{\mu_2^8},\\[-0.2em]
    & q_1 = \frac{2 m_1 \left(6 m_1^2 \mu_2^5 \mu_3-2 m_1^3 \mu_2^6-2 \mu_2^6 \mu_3+9 m_1^3 \mu_2^3 \mu_3^2-6 m_1 \mu_2^4 \mu_3^2-6
   m_1^2 \mu_2^2 \mu_3^3-2 m_1^3 \mu_3^4\right)}{\mu_2^8}, \\[-0.2em]
   & q_2 = 1+\frac{m_1 \left(19 m_1^3 \mu_2^6-6 m_1 \mu_2^7-24 m_1^2 \mu_2^5 \mu_3+4 \mu_2^6 \mu_3-18 m_1^3 \mu_2^3 \mu_3^2+6 m_1
   \mu_2^4 \mu_3^2+4 m_1^2 \mu_2^2 \mu_3^3+m_1^3 \mu_3^4\right)}{\mu_2^8},\\[-0.2em]
   & q_3 = \frac{3 m_1^2 \left(2 \mu_2^4-9 m_1^2 \mu_2^3+4 m_1 \mu_2^2 \mu_3+2 m_1^2 \mu_3^2\right)}{\mu_2^5}, \qquad
   q_4 = \frac{12 m_1^4}{\mu_2^2}.
\end{align*}
\end{corollary}
\begin{remark}
By definition we put $\binom{-2}{-2}=\binom{-1}{-1}=1$, $\binom{-1}{-2}=-1$ and $\binom{j}{-2}=\binom{j}{-1}=0, j \geq 0$.
\end{remark}
\begin{example}[General Gaussian distribution]
If $X_{ij} \sim N(\mu,\sigma^2)$, we have $m_1 = \mu$, $(\mu_2,\mu_3,\mu_4) = (\sigma^2,0,3 \sigma^4)$, $(q_{-2},q_{-1},q_0,q_1,q_2,q_3,q_4) = \left(0,0,0,-4 \mu ^4,19 \mu ^4-6 \mu ^2 \sigma ^2+\sigma ^4,6 \mu ^2 \sigma ^2-27 \mu ^4,12 \mu ^4\right)/\sigma^4$, from which we get
\begin{equation}
f_4(n)=\frac{1}{2} (n!)^2 (1+n) \sigma ^{4 (n-1)} \left(n^3 \mu ^4+(2+n) \sigma ^2 \left(2 n \mu ^2+\sigma ^2\right)\right).
\end{equation}
\end{example}
\begin{example} ($(0,2)$ matrices). Let $X_{ij} = 0,2$ with equal probability, thus $(m_1,m_2,m_3,m_4) = (1,2,4,8)$ and $(\mu_2,\mu_3,\mu_4) = (1,0,1)$. As pointed out by Terence Tao \cite{TerryTao}, the determinant of a random $n \times n$ $(-1,+1)$ matrix is equal to the determinant of a random $n-1 \times n-1$ $(0,2)$ matrix for which $(m_1,m_2,m_3,m_4)=(0,1,0,1)$. In terms of generating functions, that means
\begin{equation}
F_4(t) = \frac{\partial }{\partial t}\left(t \, \frac{\partial F_4^{\mathrm{sym}}(t)}{\partial t}\right) = \frac{\partial }{\partial t}\left(t \, \frac{\partial}{\partial t} \frac{e^{-2 t}}{(1-t)^3}\right) = \frac{e^{-2 t} \left(1+5t+2t^2 + 4 t^3\right)}{(1-t)^5},
\end{equation}
where in $F_4^{\mathrm{sym}}(t)$ we put $(m_1,m_2,m_3,m_4)=(0,1,0,1)$. This result coincides exactly with our general formula for $F_4(t)$ with $(m_1,m_2,m_3,m_4) = (1,2,4,8)$.
\end{example}
\begin{example}[Exponential distribution] If $X_{ij} \sim \operatorname{Exp}(1)$, that is if $m_j = j!$, we have $(\mu_2,\mu_3,\mu_4) = (1,2,9)$ and $(q_{-2},q_{-1},q_0,q_1,q_2,q_3,q_4) = (16, -96, 192, -124, -26, 27, 12)$. 
Using \textit{Mathematica}, we get an asymptotic behaviour for large $n$,
\begin{equation}
    f_4(n) \approx \frac{1}{2} e^6 (n!)^2 \left(450+141 n-27 n^2-5 n^3+n^4\right).
\end{equation}
The first ten exact moments are shown in Table \ref{ExpoTabl} below.

\begin{table}[h]
\centering
\begin{tabular}{|
      p{2em}|
      >{\centering}p{1.2em}|
      >{\centering}p{2.5em}|
      >{\centering}p{3em}|
      >{\centering}p{4.3em}|
      >{\centering}p{5.5em}|
      >{\centering}p{6.5em}|
      >{\centering\arraybackslash}p{8em}
      |}
  \hline
$n$ & $1$ & $2$ & $3$ & $4$ & $5$ & $6$ & $7$\\\hline
$f_4(n)$ & $24$ & $960$ & $51840$ & $3511872$ & $287953920$ & $27988001280$ & $3181325414400$ \\
\hline
    \end{tabular} 
\begin{tabular}{|
      p{2em}|
      >{\centering}p{9.2em}|
      >{\centering}p{11.05em}|
      >{\centering\arraybackslash}p{15.7em}
      |}
\hline
$n$ & $8$ & $9$ & $10$ \\ \hline
$f_4(n)$ & $418846663065600$ & $63399549828464640$ & $10964925305310412800$ \\
\hline
\end{tabular}
\caption{Fourth moment of a random determinant with entries exponentially distributed}
\label{ExpoTabl}
\end{table}
These tabulated numbers are of particular interest in the field of random geometry. Let us have a $d$-dimensional simplex with unit $d$-volume, from which we select $d+1$ points uniformly and independently. A convex hull of those random points forms a smaller simplex with $d$-volume $V_d$, which is now a random variable. As shown by Reed \cite{reed1974random}, the even moments of $V_d$ are given by
\begin{equation}
    \Ex V_d^{2l} = \left(\frac{d!}{(d+2l)!}\right)^{d+1} f_{2l}(d+1),
\end{equation}
our result applied on $X_{ij} \sim \operatorname{Exp}(1)$ thus implies an explicit formula for the fourth moment of $V_d$.
\end{example}

\begin{theorem}\label{MainThmDembo}
Defining $p_k$ and $\mu_j$ as above, we have
\begin{equation}
    F_4(t,\omega) = \frac{e^{t(\mu_4 - 3 \mu_2^2)}}{\left(1- \mu_2^2 t\right)\! {}^4 \left(1-\omega -\mu_2^2 t\right)} \left(1+ \sum_{k=1}^6 p_k t^k + \frac{\omega  m_1^2}{1-\omega- \mu_2^2 t}\sum_{k=1}^4 \tilde{p}_k t^k +\frac{2 \omega^2 m_1^4 \mu_2^2 t^2}{\left(1-\omega - \mu_2^2 t\right){}^2}\right),
\end{equation}
where
\begin{equation*}
\tilde{p}_1 = m_1^2+2 \mu_2, \,\,\, \tilde{p}_2 = 5 m_1^2 \mu _2^2+4 m_1 \mu_2 \mu_3 -2 \mu_2^3, \,\,\, \tilde{p}_3 = 2 m_1^2 \mu_2^4-4 m_1 \mu_2^3 \mu_3 +2 m_1^2 \mu_2 \mu_3^2, \,\,\, \tilde{p}_4 = -2 m_1^2 \mu_2^3 \mu_3^2.
\end{equation*}
\end{theorem}

\begin{remark}
Letting $\omega = 0$, we recover $F_4(t)$. On the other hand, letting $m_1 = 0$, we get $F^{\mathrm{sym}}_4(t,\omega)$.
\end{remark}

\noindent
\begin{corollary} Defining $q_i$ and $\mu_j$ as above, we get, by Taylor expansion,
    \begin{equation}
        f_4(n,p) = p!^2 \binom{n}{p} \mu_2^{2p} \sum_{j=0}^p \frac{1}{j!} \left(\frac{\mu_4}{\mu_2^2} - 3\right)^j \sum_{i=-2}^4 (q_i+\tilde{q}_i(n-p)+ \vardbtilde{q}_i(n-p)(n-p+7))\binom{n-j+i}{n-p+i},
    \end{equation}
where
\begin{equation*}
\begin{split}
    & \tilde{q}_0 = -\frac{2 m_1^4 \mu _3^2}{\mu _2^5}, \quad
    \tilde{q}_1 = \frac{2 m_1^3 \left(2 \mu _2^2 \mu _3+3 m_1 \mu _3^2-m_1 \mu _2^3\right)}{\mu _2^5}, \quad
    \tilde{q}_2 = \frac{m_1^2 \left(3 m_1^2 \mu _2^3-2 \mu _2^4-8 m_1 \mu _2^2 \mu _3-6 m_1^2 \mu _3^2\right)}{\mu _2^5},\\[-0.2em]
    & \tilde{q}_3 = \frac{m_1^2 \left(2 \mu _2^4+4 m_1 \mu _2^2 \mu _3+2 m_1^2 \mu _3^2-m_1^2 \mu _2^3\right)}{\mu _2^5}, \qquad \vardbtilde{q}_2 = \frac{m_1^4}{\mu _2^2}, \qquad \vardbtilde{q}_3=-\frac{2 m_1^4}{\mu _2^2}, \qquad \vardbtilde{q}_4=\frac{m_1^4}{\mu _2^2}
\end{split}    
\end{equation*}
and $\tilde{q}_i$, $\vardbtilde{q}_i$ otherwise zero.
\end{corollary}

\begin{example}[General Gaussian distribution]
If $X_{ij} \sim N(\mu,\sigma^2)$, we have $m_1 = \mu$, $(\mu_2,\mu_3,\mu_4) = (\sigma^2,0,3 \sigma^4)$, which gives, after series of simplifications,
\begin{equation}
    f_4(n,p) = \frac{n!(n+1)! \sigma ^{4 (p-1)}}{(n-p)! (n-p+2)!} \left(n p^2 \mu ^4+ (n+2)\left( 2 p \mu ^2 \sigma ^2+ \sigma ^4\right)\right).
\end{equation}
This formula agrees with the general case given by Equation \ref{ThmGauss}.
\end{example}

\begin{example}[Exponential distribution] If $X_{ij} \sim \operatorname{Exp}(1)$, that is if $m_j = j!$, we have $(\mu_2,\mu_3,\mu_4) = (1,2,9)$ and $(q_{-2},q_{-1},q_0,q_1,q_2,q_3,q_4,\tilde{q}_0,\tilde{q}_1,\tilde{q}_2,\tilde{q}_3,\vardbtilde{q}_2,\vardbtilde{q}_3,\vardbtilde{q}_4)\! =\! (16,\! -96,\! 192,\! -124,\! -26,\! 27,\! 12,\! -8,\! 30,\! -39,\! 17,\! 1,\! -2,\! 1)$. 
The exact moments $f_4(n,p)$ for low $n$ and $p$ are shown in Table \ref{Expo2Tabl} below.
\begin{table}[h]
\centering
\begin{tabular}{|
      >{\centering}p{2.6em}|
      >{\centering}p{0.5em}|
      >{\centering}p{1.6em}|
      >{\centering}p{3.6em}|
      >{\centering}p{4.7em}|
      >{\centering}p{6.6em}|
      >{\centering}p{7.9em}|
      >{\centering\arraybackslash}p{10.3em}
      |}
\hline
 \multicolumn{2}{|c|}{\multirow{2}{*}{$f_4(n,p)$}} & \multicolumn{6}{c|}{$p$}  \\ \cline{3-8}
   \multicolumn{2}{|c|}{}  & $1$ & $2$ & $3$ & $4$ & $5$ & $6$  \\
\hline
\multirow{8}{*}{$n-p$} & $0$ & $24$ & $960$ & $51840$ & $3511872$ & $287953920$ & $27988001280$ \\
 & $1$ & $56$ & $3744$ & $297216$ & $27708480$ & $3004024320$ & $375698373120$ \\
 & $2$ & $96$ & $9432$ & $1022400$ & $124675200$ & $17182609920$ & $2675406827520$ \\
 & $3$ & $144$ & $19320$ & $2724480$ & $419207040$ & $71341240320$ & $13491506810880$ \\
 & $4$ & $200$ & $34920$ & $6189120$ & $1169602560$ & $240336875520$ & $54144163584000$ \\
 & $5$ & $264$ & $57960$ & $12579840$ & $2858913792$ & $696776048640$ & $184099283343360$ \\
 & $6$ & $336$ & $90384$ & $23538816$ & $6325119360$ & $1801876285440$ & $551197391754240$ \\
 & $7$ & $416$ & $134352$ & $41299200$ & $12939696000$ & $4256462960640$ & $1491202996208640$ \\
\hline
\end{tabular}
\caption{Second moment of a random Gram determinant with entries exponentially distributed}
\label{Expo2Tabl}
\end{table}
\end{example}

\clearpage
\section{Proof of Theorem \ref{MainThm}}
\subsection{NRR's generating function}
We briefly discuss what we believe is a simpler derivation of $F^{\mathrm{sym}}_4(t)$ of Nyquist, Rice and Riordan \cite{nyquist1954distribution}. We were inspired by the paper of Lv and Potechin \cite{LP2022}.

\begin{lemma}\label{LemFlajo}
Let $S_n$ be the set of all permutations of order $n$ and $D_n$ the set of all \textbf{derangements} of the same order (that is, $D_n$ is a subset of those permutations in $S_n$ which have no fixed points). Denote $C(\pi)$ the number of cycles in a permutation $\pi$, then
\begin{equation}
\sum_{n=0}^\infty \frac{x^n}{n!}\sum_{\pi \in D_n} u^{C(\pi)} = \frac{e^{-ux}}{(1-x)^u}.
\end{equation}
\end{lemma}
\begin{proof}
See \cite{flajolet2009analytic}, chapter on Bivariate generating function.
\end{proof}

\begin{proposition}
\begin{equation}
f_4(n) = \Exx |A|^2 = \Exx \sum_{\pi_1,\pi_2,\pi_3,\pi_4 \in S_n} \prod_{r=1}^4 \left( \operatorname{sgn}\pi_r \prod_{i=1}^n X_{i\pi_r(i)}\right).
\end{equation}
\end{proposition}
\begin{proof}
Follows from the definition of determinant.
\end{proof}

The summation above is carried over all permutation fours $(\pi_r)_{r=1}^4$, which can be viewed, the same way as it is in the original article of Nyquist, Rice and Riordan \cite{nyquist1954distribution}, as a sum over all possible \emph{permutation tables}. Since we assume $X_{ij}$ follow symmetric distribution, many of the terms vanish.

\begin{definition}[Permutation tables]
We say $t$ is a permutation four-table of length $n$ if its rows are exactly the permutations $\pi_r$ of length $n$. A table $t$ is called \textit{symmetric}, if its columns fall into the admissible categories below. Furthermore, we assign \textit{weight} to each column. The weight $w(t)$ of the table $t$ is then simply a product of weights of its columns. Similarly we define sign of a table as a product of sign of the permutations in each row. The admissible columns in symmetric four-tables are:
\begin{itemize}
    \item 4-columns: four copies of a single number (weight $m_4$)
    \item 2-columns: two pairs of distinct numbers (weight $m_2^2$)
\end{itemize}
We denote $T^{\mathrm{sym}}_{4,n}$ the set of all symmetric four-tables of length $n$.
\end{definition}
\begin{remark}
To distinguish between tables, we sometimes write $t_r$ instead of $\pi_r$ for the rows of $t$.
\end{remark}

\begin{proposition}
\begin{equation}
    f^{\mathrm{sym}}_4(n) = \sum_{t \in T^{\mathrm{sym}}_{4,n}} w(t)\operatorname{sign}(t),
\end{equation}
\end{proposition}
\begin{proof}
Follows from the definitions above.
\end{proof}

We group the summands according to number of 2-columns in $t$. Those columns form a subtable $s$ and the rest of the columns form another, a complementary subtable $t'$. The signs of those tables are related as
\begin{equation}
 \operatorname{sgn}(t) = \operatorname{sgn}(s) \operatorname{sgn}(t').
\end{equation}
Denote $\left[n\right] = \{1,2,3,\ldots,n\}$. For a given $J \subset \left[n\right]$, we define $T^\mathrm{sym}_{4,J}$ a set of all symmetric four-tables of length $j = |J|$ composed with numbers in $J$. The set $T^\mathrm{sym}_{4,n}$ coincide then with $T^\mathrm{sym}_{4,\left[n\right]}$. Denote $D_{4,J}$ the set of all four-tables composed only from 2-columns of numbers in $J$. We can write our sum, since the selection $J$ does not depend on position in table $t$, as
\begin{equation}
    f^{\mathrm{sym}}_4(n) = \sum_{J \subset \left[n\right]} \binom{n}{j} \sum_{t' \in T^{\mathrm{sym}}_{4,\left[n\right]/J}} w(t)\operatorname{sgn}(t) \sum_{s \in Q_{4,J}} w(s)\operatorname{sgn}(s).
\end{equation}
No matter which numbers $J$ are selected, as long as we select the same amount of them, the contribution is the same. Hence,
\begin{equation}
    f^{\mathrm{sym}}_4(n) = \sum_{j = 0}^n \binom{n}{j}^2 \sum_{t' \in T^{\mathrm{sym}}_{4,n-j}} w(t)\operatorname{sgn}(t) \sum_{s \in D_{4,j}} w(s)\operatorname{sgn}(s),
\end{equation}
where $Q_{4,j} = Q_{4,\left[j\right]}$. For the first inner sum, notice that table $t'$ is composed of only four-columns, so $w(t') = m_4^{n-j}$ and $\operatorname{sgn}(t') = (\pm 1)^4 = 1$. Also note that $|T^\mathrm{sym}_{4,n-j}| = (n-j)!$ . For the second inner sum, by symmetry, we can fix the first permutation in $s$ to be identity, giving us the factor of $j!$ . Upon noticing also that $w(s) = m_2^{2j}$ and $\operatorname{sgn}(s) = (\pm 1)^2 = 1$, we get,
\begin{equation}
f^{\mathrm{sym}}_4(n) = \sum_{j = 0}^n \binom{n}{j}^2 (n-j)! m_4^{n-j} \, j! m_2^{2j} \sum_{\substack{s \in D_{4,j} \\ s_1 = \mathrm{id}}} 1.
\end{equation}
We group the summands according to the following permutation structure: Let $b$ be a number in the first row of a given column of table $s$. Since it is a 2-column, we denote the other number in the column as $b'$. We construct a permutation $\pi(s)$ to a given table $s$ as composed from all those pairs $b \rightarrow b'$. Note that since $b$ and $b'$ are allways different, the set off all $\pi(s)$ corresponds to the set $D_j$ of all derangements. Since there are $3$ possibilities how to arrange the leftover 3 numbers in the 2-columns corresponding to a given cycle of $\pi(s)$, we get
\begin{equation}
\sum_{\substack{s \in D_{4,j} \\ s_1 = \mathrm{id}}} 1 = \sum_{\pi \in D_j} 3^{C(\pi)}.
\end{equation}
Hence, in terms of generating functions,
\begin{equation}
    F_4^{\mathrm{sym}}(t) \! = \!\sum_{n=0}^\infty \frac{t^n}{n!^2} f^{\mathrm{sym}}_4(n) \! = \! \sum_{n=0}^\infty\sum_{j = 0}^n \frac{\left(m_4 t\right)^{n-j}}{(n-j)!} \frac{(m_2^2t)^j}{j!} \!\!\sum_{\pi \in D_j}\!\! 3^{C(\pi)} \! =\! e^{m_4t} \sum_{j=0}^\infty\frac{(m_2^2t)^j}{j!} \!\!\sum_{\pi \in D_j}\!\! 3^{C(\pi)} \! =\!  e^{m_4t} \frac{e^{-3m_2^2t}}{(1\! -\! m_2^2t)^{3}}.
\end{equation}
The final equality is a special case of Lemma \ref{LemFlajo}.

\subsection{Matrix determinant lemma}
The proof of Theorem \ref{MainThm} relies on the fact that $f^{\mathrm{cen}}_4(n) =  f^{\mathrm{sym}}_4(n)$ combined with the following key lemma:
\begin{lemma}\label{MainLem}
Let $C = (c_{ij})_{n \times n}$ be any real matrix, $u = (u_i)_{n \times 1}$, $v=(v_i)_{n \times 1}$ real vectors and $\lambda \in \mathbb{R}$, then
\begin{equation}
    |C+\lambda u v^T| = |C| + \lambda v^T C^{\mathrm{adj}} u,
\end{equation}
where $(C^{\mathrm{adj}})_{ij} = (-1)^{i+j} |C_{ji}|$ is called the \textbf{adjugate matrix} of $C$ and $C_{ji}$ denotes a matrix formed from $C$ by deleting its $j$-th row and $i$-th column, as usual.
\end{lemma}
\begin{proof}
In fact, the lemma is a special case of the \textbf{Weinstein–Aronszajn identity}. To see this, consider
\begin{equation}
|C+\lambda uv^T| = |C|\, |I + \lambda C^{-1} u v^T| = |C|\, |I + \lambda v^T C^{-1} u| = |C|\left(1+\lambda v^T C^{-1} u\right) = |C| + \lambda v^T C^{\mathrm{adj}} u.
\end{equation}
By continuity, we conclude that the lemma holds even for $C$ being noninvertible.
\end{proof}

\begin{definition}[$Y_{ij},\mu_r$]
We denote $Y_{ij} = X_{ij} - m_1$ and $\mu_r = \Ex Y_{ij}^r$.
\end{definition}
\begin{remark}
Clearly, $Y_{ij}$'s are \textbf{centered} i.i.d. random variables with moments depending on $m_j$ as such
\begin{equation}
    \mu_1 = 0, \qquad
    \mu_2 = m_2-m_1^2, \qquad
    \mu_3 = m_3-3 m_1 m_2+2 m_1^3, \qquad
    \mu_4 = m_4-4m_1m_3+6 m_1^2 m_2-3 m_1^4,
\end{equation}
and so on.
\end{remark}
\begin{definition}[$B,g_k(n),G_k(t)$] Given $Y_{ij}$'s, we form a matrix $B=(Y_{ij})_{n \times n}$ and denote $g_k(n) = \Exx |B|^k$ and
\begin{equation}
    G_k(t) = \sum_{n=0}^\infty \frac{t^n}{(n!)^2} g_k(n).
\end{equation}
\end{definition}
\begin{remark}
Since the moments of a random determinant are dependent only on moments of its random entries, we get that $g_k(n)$ is equal to  $f_k^{\mathrm{cen}}(n)$ in which we replace $m_r$ by $\mu_r$. So, for $k=4$,
\begin{equation}
    G_4(t) = \frac{e^{t(\mu_4-3\mu_2^2)}}{(1-\mu_2^2t)^3}.
\end{equation}
\end{remark}

\begin{proposition}
\begin{equation}
    |A| = |B| + m_1 S,\qquad \text{where} \qquad S = \sum_{ij} (-1)^{i+j} |B_{ij}|.
\end{equation}
\end{proposition}
\begin{proof}
By definition of $Y_{ij}$'s and $B$, we can write
\begin{equation}
    A = B + m_1 uu^T,
\end{equation}
where $u$ is a column vector with $n$ rows having all components equal to one. Hence, by Lemma \ref{MainLem},
\begin{equation}
|A| = |B + m_1 u u^T| = |B| + m_1 u^T B^{\mathrm{adj}} u = |B| + m_1 \sum_{ij} u_i (-1)^{i+j} |B_{ji}| \, u_j = |B| + m_1 S.
\end{equation}
\end{proof}
\begin{corollary}
\label{ColSummands} We thus get an expression for $f_4(n)$ in terms of the following \textbf{summands}
\begin{equation}
f_4(n) = \Exx |A|^4 = \Exx \left(|B|+m_1 S\right)^4 = \Exx|B|^4 + 4 m_1 \Exx |B|^3 S + 6 m_1^2 \Exx |B|^2 S^2 + 4 m_1^3 \Exx |B| S^3 + m_1^4 \Exx S^4. 
\end{equation}
\end{corollary}
\begin{remark}
The first summand is trivial, since we already know that $\Ex |B|^4 = g_4(n)$. The goal of the rest of our paper is to express the other summands in terms of $g_4(n)$ as well. This was be possible due to the crucial fact that $B$ now has only \textbf{centered} random entries $Y_{ij}$'s. The main tool to obtain such relations is using the \textbf{Laplace expansion} of determinants via their rows (or columns) repeatedly.
\end{remark}

\begin{definition}[$B_{ij,kl}$,matrix symbols]
We denote $B_{ij,kl}$ a matrix $B$ from which the rows $i,j$ and columns $k,l$ were deleted. To improve readability, we adopt a graphical notation (\textbf{matrix symbols}) for determinants $|B_{ij}|$ and $|B_{ij,kl}|$. We write, for example,
\begin{equation*}
    \mBbb = |B_{22}|, \qquad\qquad\mBbcbd = |B_{23,24}|,\qquad\qquad \mBab = \mBabic = |B_{12}|.
\end{equation*}
A row painted in black (or a column as in the example above) shows where the Laplace expansion is being performed in the next step.
\end{definition}
\begin{remark}
Table \ref{AllmB} in the appendix shows all matrix symbols used in this paper.
\end{remark}

\subsection{Second summand}
\begin{proposition}\label{Lem2}
\begin{equation}
    \Exx |B|^3 S = n^2 \mu_3 g_4(n-1).
\end{equation}
\end{proposition}
\begin{proof}
By symmetry, $\Exx |B|^3 S = n^2 \Exx |B|^3 |B_{11}|$, that is
\begin{equation}
\begin{split}
\Exx |B|^3 S = n^2 \Ex \left(\mBaj\right)^3 \mBaa = n^2 \mu_3 \Ex \left(\mBaa\right)^4 - n^2(n-1) \mu_3 \Ex \left(\mBab\right)^3 \mBaaib  = n^2 \mu_3 g_4(n-1).
\end{split}
\end{equation}
\end{proof}

\subsection{Third summand}
\begin{proposition}\label{Lem3}
\begin{equation}
    \Exx |B|^2 S^2 = n^2 h_0(n)+n^2(n-1)^2 \mu_3^2 g_4(n-2),
\end{equation}
where $h_0(n)$ satisfies the recurrence relation
\begin{equation}
    h_0(n) = \mu_2 g_4(n-1) + (n-1)^2 \mu_2^2 h_0(n-1).
\end{equation}
\end{proposition}
\begin{proof}
By definition of $S$, we have
\begin{equation}
    \Ex\, |B|^2 S^2 = \sum_{ijkl} (-1)^{i+j+k+l} \,\Exx|B|^2|B_{ij}||B_{kl}|.
\end{equation}

The terms $\Ex |B|^2 |B_{ij}| |B_{kl}|$ in the sum above form equivalence \textbf{classes} in which each member has the same contribution (up to a sign). Each class is characterised by having the same relative arrangement of pairs of indices $(ij)$ and $(kl)$ in the $n \times n$ matrix grid. \textbf{Representants} drawn from each class together with their \textbf{signs} and \textbf{values} denoted $h_i(n)$ are shown in Table \ref{TaSum3} below (the diagrams represent the relative arrangement of indices for a given representant). The table also shows the total number of terms in the same equivalence class (\textbf{size} of a class).

\renewcommand{\arraystretch}{1.5}
   \begin{table}[h]
\centering
\setlength{\tabcolsep}{2pt} 
\setlength\jot{1pt} 
\begin{tabular}{|c|c|c|c|}
 \hline
 \tupsingle \begin{tabular}{c} sign \\ class \\ size \\ value \end{tabular} &
\begin{tabular}{c} $+$ \\ $\mBudef{\mCaaaa}$ \\ $n^2$ \\ $h_0(n)$ \end{tabular} &
\begin{tabular}{c} $-$ \\ $\mBudef{\mCaaab}$ \\ $2n^2(n-1)$ \\ $h_1(n)$ \end{tabular} &
\begin{tabular}{c} $+$ \\ $\mBudef{\mCaabb}$ \\ $n^2(n-1)^2$ \\ $h_2(n)$
\end{tabular} \\
\hline
\end{tabular} 
\caption{Classes of equivalent terms in the third summand}
\label{TaSum3}
\end{table}

\FloatBarrier
\noindent
Thus,
\begin{equation}
 \Ex\, |B|^2 S^2 = n^2 h_0(n) - 2n^2(n-1)h_1(n) + n^2(n-1)^2 h_2(n)
\end{equation}
with
\begin{equation}
    h_0(n) = \Exx |B|^2 |B_{11}|^2, \qquad h_1(n) = \Exx |B|^2|B_{11}||B_{12}|, \qquad h_2(n) = \Exx |B|^2 |B_{11}||B_{12}|.
\end{equation}
We shall now perform the Laplace expansion on those terms until we get a recurrence relation,
\begin{align}
\begin{split}
    h_0(n) & = \Exx |B|^2 |B_{11}|^2 = \Ex \left(\mBaj\right)^2\left(\mBaa\right)^2 = \mu_2 \Exx |B_{11}|^4 + (n-1)\mu_2\Ex \left(\mBabia\right)^2\left(\mBaa\right)^2 = \\
        & = \mu_2 g_4(n-1) + (n-1)^2 \mu_2^2 \Ex \left(\mBabab\right)^2\left(\mBaa\right)^2 = \mu_2 g_4(n-1) + (n-1)^2 \mu_2^2 h_0(n-1),
\end{split}\\
\begin{split}
    h_1(n) & = \Exx |B|^2|B_{11}||B_{12}| = \Ex \left(\mBaj\right)^2\mBaa \mBab = \\
        & = \mu_2 \Ex \left(\mBaa\right)^3\mBabia + \mu_2 \Ex \left(\mBab\right)^3\mBaaib + (n-2)\mu_2\Ex \left(\mBacia\right)^2\mBaa\mBabia = \\
        & = (n-1)(n-2)\mu_2\mu_3 \Ex \left(\mBabac\right)^2\mBaabj \mBabab = 0,
\end{split}\\
\begin{split}
    h_2(n) & =  \Exx |B|^2 |B_{11}||B_{12}| = \Ex \left(\mBaj\right)^2\mBaa\mBbbaj = \\
        & = \mu_3 \Ex\left(\mBaa\right)^3 \mBabab - (n-2)\mu_3 \Ex \left(\mBac\right)^2 \mBaaic \mBabbc = \mu_3 \left[ \Exx |B|^3|B_{11}|\right]_{n\rightarrow n-1} = \mu_3^2 g_4(n-2).
\end{split}
\end{align}
\end{proof}

\subsection{Fourth summand}
\begin{proposition}\label{Lem4}
\begin{equation}
    \Ex\, |B| S^3 = 3n^2(n-1)^2 \mu_3 h_0(n-1)+n^2(n-1)^2(n-2)^2 \mu_3^3 g_4(n-3).
\end{equation}
\end{proposition}
\begin{proof}
\begin{equation}
    \Exx |B| S^3 = \sum_{ijklrs} (-1)^{i+j+k+l+r+s} \,\Exx|B||B_{ij}||B_{kl}||B_{rs}|. 
\end{equation}

The following Table \ref{TaSum4} summarizes all the possible classes of terms according to the arrangement of $(ij),(kl),(rs)$ indices. Note that there are some representats whose value is trivially zero (they contain a row or a column such that the expansion in which gives zero). Thus
\begin{equation}
    \Exx |B| S^3 = 3n^2(n-1)^2 h_3(n) + 6n^2(n-1)^2 h_4(n) + 6n^2(n-1)^2 (n-2)h_5(n) + n^2(n-1)^2(n-2)^2 h_6(n).
\end{equation}

\FloatBarrier
\renewcommand{\arraystretch}{1.5}
   \begin{table}[h]
\centering
\setlength{\tabcolsep}{2pt} 
\setlength\jot{1pt} 
\begin{tabular}{|c|c|c|c|c|c|c|c|}
 \hline
 \tupsingle \begin{tabular}{c} sign \\ class \\ size \\ value \end{tabular} &
\begin{tabular}{c} $+$ \\ $\mBudef{\mCaaaaaa}$ \\ ${\scriptstyle n^2}$ \\ $0$ \end{tabular} &
\begin{tabular}{c} $-$ \\ $\mBudef{\mCaaaaab}$ \\ ${\scriptstyle 6n^2(n-1)}$ \\ $0$ \end{tabular} &
\begin{tabular}{c} $+$ \\ $\mBudef{\mCaaaabb}$ \\ ${\scriptstyle 3n^2(n-1)^2}$ \\ $h_3(n)$ \end{tabular} &
\begin{tabular}{c} $+$ \\ $\mBudef{\mCaaabba}$ \\ ${\scriptstyle 6n^2(n-1)^2 }$ \\ $h_4(n)$ \end{tabular} &
\begin{tabular}{c} $+$ \\ $\mBudef{\mCaaacbb}$ \\ ${\scriptstyle 6n^2(n-1)^2(n-2)}$ \\ $h_5(n)$ \end{tabular} &
\begin{tabular}{c} $+$ \\ $\mBudef{\mCaabbcc}$ \\ ${\scriptstyle n^2(n-1)^2(n-2)^2 }$ \\ $h_6(n)$ \end{tabular} &
\begin{tabular}{c} $-$ \\ $\mBudef{\mCaaabac}$ \\ ${\scriptstyle 2n^2(n-1)(n-2)}$ \\ $0$ \end{tabular} \\
\hline
\end{tabular} 
\caption{Classes of equivalent terms in the fourth summand}
\label{TaSum4}
\end{table}

\FloatBarrier

We now proceed to expand the values of the nontrivial representants until we get recurrence relations,
\begin{align}
\begin{split}
    h_3(n) & = \Exx |B| |B_{11}|^2 |B_{22}| = \Exx \mBib \left(\mBaaib\right)^2 \mBbb = \\
        & = \mu_3 \Ex\left(\mBbb\right)^2 \left(\mBabab\right)^2 - (n-2)\mu_3 \Exx \mBcb\left(\mBacab\right)^2\mBbbcj = \\
        & = \mu_3 \left[\Exx |B|^2|B_{11}|^2\right]_{n \rightarrow n-1} = \mu_3 h_0(n-1),
\end{split}\\
\begin{split}
    h_4(n) & = \Exx |B||B_{12}||B_{21}||B_{22}| = \Exx \mBaj\mBab\mBbaaj\mBbbaj = (n-2)\mu_3 \Exx \mBac\mBabic \mBabac \mBabbc = 0,
\end{split}\\
\begin{split}
    h_5(n) & = \Exx |B||B_{11}||B_{13}||B_{22}| = \Exx\mBia\mBaa\mBacia\mBbbia = (n-2)\mu_3 \Exx \mBca\mBaacj\mBacac\mBbcab = 0,
\end{split}\\
\begin{split}
    h_6(n) & = \Exx |B||B_{11}||B_{22}||B_{33}| = \Exx \mBaj \mBaa\mBbbaj \mBbbaj = \\
        & = \mu_3\Ex\left(\mBaa\right)^2\mBabab\mBacac +(n-3)\mu_3 \Exx \mBad\mBaaid\mBabbd\mBaccd = \\
        & = \mu_3\left[\Exx|B|^2|B_{11}||B_{22}|\right]_{n\rightarrow n-1} = \mu_3 h_2(n-1) = \mu_3^3 g_4(n-3).
\end{split}    
\end{align}
\end{proof}

\subsection{Fifth summand}
\begin{proposition}\label{Lem5}
\begin{equation}
\begin{split}
    \Ex S^4 & = n^2 g_4(n-1) + 6n^2 (n-1)^2 \mu_2 h_0(n-1) + 3n^2 (n-1)^2 h_9(n) + 6 n^2 (n-1)^2 h_{10}(n)\, + \\
    & + 6 n^2 (n-1)^2 (n-2)^2 \mu_3^2 h_0(n-2) + n^2 (n-1)^2 (n-2)^2 (n-3)^2 \mu_3^4 g_4(n-4),
\end{split}
\end{equation}
where $h_9(n)$ and $h_{10}(n)$ satisfy the recurrence relations
\begin{equation}
\begin{split}
    h_9(n) & = \mu_2 h_0(n-1) + (n-2)^2 \mu_2^3 h_0(n-2) + (n-2)^2 \mu_2^2 h_9(n-1), \\
    h_{10}(n) & = (n-2)^2 \mu_2^3 h_0(n-2) + (n-2)^2 \mu_2^2 h_{10}(n-1).    
\end{split}
\end{equation}
\end{proposition}
\begin{proof}
\begin{equation}
    \Exx S^4 = \sum_{ijklrsuv} (-1)^{i+j+k+l+r+s+u+v} \,\Exx |B_{ij}||B_{kl}||B_{rs}||B_{uv}|.
\end{equation}

As in the previous cases, we have summarised the representants of all classes in Table \ref{TaSum5}. Again note that the values of some of them are trivially zero.

\FloatBarrier
\renewcommand{\arraystretch}{1.5}
   \begin{table}[h]
\centering
\setlength{\tabcolsep}{2pt} 
\setlength\jot{1pt} 

\begin{tabular}{|
      >{\centering}p{3em}|
      >{\centering}p{3em}|
      >{\centering}p{4em}|
      >{\centering}p{4.5em}|
      >{\centering}p{4.6em}|
      >{\centering}p{6.5em}|
      >{\centering}p{7em}|
      >{\centering}p{6.8em}|
      >{\centering\arraybackslash}p{7.2em}
      |}
 \hline
\tupsingle
\begin{tabular}{c} sign \\ class \\ size \\ value \end{tabular} &
\begin{tabular}{c} $+$ \\ $\mBudef{\mCaaaaaaaa}$ \\ ${\scriptstyle n^2}$ \\ $h_7(n)$ \end{tabular} &
\begin{tabular}{c} $+$ \\ $\mBudef{\mCaaaaabab}$ \\ ${\scriptstyle 6n^2(n\!-\!1)}$ \\ $h_8(n)$ \end{tabular} &
\begin{tabular}{c} $+$ \\ $\mBudef{\mCaaaabbbb}$ \\ ${\scriptstyle 3n^2(n\!-\!1)^2}$ \\ $h_9(n)$ \end{tabular} &
\begin{tabular}{c} $+$ \\ $\mBudef{\mCaaabbabb}$ \\ ${\scriptstyle 6n^2(n\!-\!1)^2 }$ \\ $h_{10}(n)$ \end{tabular} &
\begin{tabular}{c} $-$ \\ $\mBudef{\mCaaaaabac}$ \\ ${\scriptstyle 12n^2(n\!-\!1)(n\!-\!2)}$ \\ $h_{11}(n)$ \end{tabular} &
\begin{tabular}{c} $-$ \\ $\mBudef{\mCaaaabbbc}$ \\ ${\scriptstyle 12n^2(n\!-\!1)^2(n\!-\!2)}$ \\ $h_{12}(n)$ \end{tabular} &
\begin{tabular}{c} $+$ \\ $\mBudef{\mCaaaabbcc}$ \\ ${\scriptstyle 6n^2(n\!-\!1)^2(n\!-\!2)^2}$ \\ $h_{13}(n)$ \end{tabular} &
\begin{tabular}{c} $+$ \\ $\mBudef{\mCaaabbacc}$ \\ ${\scriptstyle 24n^2(n\!-\!1)^2(n\!-\!2)^2}$ \\ $h_{14}(n)$ \end{tabular} \\
\hline
    \end{tabular} 

\begin{tabular}{|
      >{\centering}p{3em}|
      >{\centering}p{6.2em}|
      >{\centering}p{6.3em}|
      >{\centering}p{8em}|
      >{\centering}p{8.5em}|
      >{\centering}p{7.5em}|
      >{\centering\arraybackslash}p{8em}
      |}
 \hline
\tupsingle
\begin{tabular}{c} sign \\ class \\ size \\ value \end{tabular} &
\begin{tabular}{c} $-$ \\ $\mBudef{\mCaaacbabb}$ \\ ${\scriptstyle \! 24n^2\!(n\!-\!1)^2\!(n\!-\!2)}$ \\ $h_{15}(n)$ \end{tabular} &
\begin{tabular}{c} $+$ \\ $\mBudef{\mCabacbaca}$ \\ ${\scriptstyle 6n^2\!(n\!-\!1)^2\!(n\!-\!2)^2}$ \\ $h_{16}(n)$ \end{tabular} &
\begin{tabular}{c} $+$ \\ $\mBudef{\mCaaabbcbd}$ \\ ${\scriptstyle \! 6n^2\!(n\!-\!1)^2\!(n\!-\!2)(n\!-\!3) }$ \\ $h_{17}(n)$ \end{tabular} &
\begin{tabular}{c} $-$ \\ $\mBudef{\mCaaabbccd}$ \\ ${\scriptstyle \! 12n^2\!(n\!-\!1)^2\!(n\!-\!2)^2\!(n\!-\!3)}$ \\ $h_{18}(n)$ \end{tabular} &
\begin{tabular}{c} $+$ \\ $\mBudef{\mCaaabacad}$ \\ ${\scriptstyle \! 2n^2\!(n\!-\!1)(n\!-\!2)(n\!-\!3)}$ \\ $h_{19}(n)$ \end{tabular} &
\begin{tabular}{c} $+$ \\ $\mBudef{\mCaabbccdd}$ \\ ${\scriptstyle n^2\!(n\!-\!1)^2\!(n\!-\!2)^2\!(n\!-\!3)^2}$ \\ $h_{20}(n)$ \end{tabular} \\
\hline
    \end{tabular} 

\begin{tabular}{|
      >{\centering}p{3em}|
      >{\centering}p{4.55em}|
      >{\centering}p{5.1em}|
      >{\centering}p{5.2em}|
      >{\centering}p{5.2em}|
      >{\centering}p{7.5em}|
      >{\centering}p{7.5em}|
      >{\centering\arraybackslash}p{9em}
      |}
 \hline
\tupsingle
\begin{tabular}{c} sign \\ class \\ size \\ value \end{tabular} &
\begin{tabular}{c} $-$ \\ $\mBudef{\mCaaaaaaab}$ \\ ${\scriptstyle 8n^2(n\!-\!1)}$ \\ $0$ \end{tabular} &
\begin{tabular}{c} $+$ \\ $\mBudef{\mCaaaaaabb}$ \\ ${\scriptstyle 4n^2(n\!-\!1)^2}$ \\ $0$ \end{tabular} &
\begin{tabular}{c} $+$ \\ $\mBudef{\mCaaaaabba}$ \\ ${\scriptstyle 12n^2(n\!-\!1)^2}$ \\ $0$ \end{tabular} &
\begin{tabular}{c} $-$ \\ $\mBudef{\mCaaaaabbb}$ \\ ${\scriptstyle 24n^2(n\!-\!1)^2 }$ \\ $0$ \end{tabular} &
\begin{tabular}{c} $+$ \\ $\mBudef{\mCaaaaacbb}$ \\ ${\scriptstyle 24n^2(n\!-\!1)^2(n\!-\!2)}$ \\ $0$ \end{tabular} &
\begin{tabular}{c} $+$ \\ $\mBudef{\mCaaabacba}$ \\ ${\scriptstyle 24n^2(n\!-\!1)^2(n\!-\!2)}$ \\ $0$ \end{tabular} &
\begin{tabular}{c} $-$ \\ $\mBudef{\mCaaabacbd}$ \\ ${\scriptstyle 8n^2(n\!-\!1)^2(n\!-\!2)(n\!-\!3)}$ \\ $0$ \end{tabular} \\
\hline
\end{tabular} 
\caption{Classes of equivalent terms in the fifth summand}
\label{TaSum5}
\end{table}

\FloatBarrier

Employing the Laplace expansion on the nontrivial terms, we obtain
\begin{align}
\begin{split}
h_7(n) & = \Exx |B_{11}|^4 = \left[\Exx|B|^4\right]_{n\rightarrow n-1} = g_4(n-1),
\end{split}\\
\begin{split}
h_8(n) & = \Exx|B_{11}|^2|B_{12}|^2 = \Ex\left(\mBaa\right)^2\left(\mBabia\right)^2 = (n-1)\mu_2 \Ex\left(\mBaa\right)^2\left(\mBabab\right)^2 = (n-1)\mu_2 h_0(n-1),
\end{split}\\
\begin{split}
h_9(n) & = \Exx|B_{11}|^2|B_{22}|^2=\Ex\left(\mBaa\right)^2\!\!\left(\mBbbaj\right)^2 \!\!= \mu_2 \Ex\left(\mBaa\right)^2\!\!\left(\mBabab\right)^2\!\!\!+(n-2)\mu_2\Ex\left(\mBaaib\right)^2\!\!\left(\mBabbc\right)^2\!=\\
    & = \mu_2 h_0(n-1) + (n-2)\mu_2^2 \Ex\left(\mBabab\right)^2\left(\mBabbc\right)^2 + (n-2)^2\mu_2^2 \Ex\left(\mBacab\right)^2\left(\mBabbc\right)^2 = \\
    & = \mu_2 h_0(n-1) + (n-2)\mu_2^2 \left[\Exx |B_{11}|^2|B_{12}|^2\right]_{n\rightarrow n-1} + (n-2)^2\mu_2^2 \left[\Exx |B_{21}|^2|B_{12}|^2\right]_{n\rightarrow n-1} = \\
    &= \mu_2 h_0(n-1) + (n-2)\mu_2^2 h_8(n-1) + (n-2)^2\mu_2^2 \left[\Exx |B_{11}|^2|B_{22}|^2\right]_{n\rightarrow n-1} =\\
    &= \mu_2 h_0(n-1) + (n-2)^2\mu_2^3 h_0(n-2) + (n-2)^2\mu_2^2 h_9(n-1),
\end{split}\\
\begin{split}
h_{10}(n) & = \Exx |B_{11}||B_{12}||B_{21}||B_{22}| = \Exx\mBaa\mBab\mBbaaj\mBbbaj = (n-2)\mu_2\Exx \mBaaic\mBabic\mBabac\mBabbc =\\
    & = (n-2)\mu_2^2 \Ex\left(\mBabac\right)^2\left(\mBabbc\right)^2+(n-2)^2\mu_2^2\Exx\mBacac\mBacbc\mBabac\mBabbc= \\
    & = (n-2)\mu_2^2 \left[\Exx|B_{11}|^2|B_{12}|^2\right]_{n\rightarrow n-1} +(n-2)^2 \mu_2^2 \left[\Exx |B_{21}||B_{22}||B_{11}||B_{12}|\right]_{n\rightarrow n-1} =\\
    & = (n-2)^2 \mu_2^3 h_0(n-2) + (n-2)^2 \mu_2^2 h_{10}(n-1),
\end{split}\\
\begin{split}
h_{11}(n) & = \Exx |B_{11}|^2 |B_{12}||B_{13}| = \Ex\left(\mBaaib\right)^2\mBab\mBacib = - (n-1)\mu_3 \Ex\left(\mBabab\right)^2\mBabbj\mBabbc = 0,
\end{split}\\
\begin{split}
h_{12}(n) & = \Exx |B_{11}|^2|B_{22}||B_{23}| = \Ex\left(\mBaaib\right)^2\mBbb\mBbcib = (n-2)\mu_3 \Ex\left(\mBacab\right)^2\mBbbcj\mBbcbc = 0,
\end{split}\\
\begin{split}
h_{13}(n) & = \Exx |B_{11}|^2 |B_{22}||B_{33}| = \Ex\left(\mBaabj\right)^2 \mBbb\mBccbj =\\
    & = \mu_3 \Ex\left(\mBabab\right)^2\mBbb\mBbcbc - (n-3)\mu_3 \Ex \left(\mBabad\right)^2\mBbbid\mBbccd =\\
    & =\mu_3 \left[\Exx |B_{11}|^2B| |B_{22}|\right]_{n\rightarrow n-1} = \mu_3 h_3(n-1) = \mu_3^2 h_0(n-2),
\end{split}\\
\begin{split}
h_{14}(n) & = \Exx |B_{11}||B_{12}||B_{21}||B_{33}| = \Exx \mBaabj\mBabbj\mBba\mBccbj = -(n-3)\mu_3\Exx\mBabad\mBabbd\mBbaid\mBbccd = 0,
\end{split}\\
\begin{split}
h_{15}(n) & = \Exx |B_{11}||B_{21}||B_{22}||B_{13}| = \Exx\mBaaib\mBbaib\mBbb\mBacib = (n-2)\mu_3\Exx\mBacab\mBbcab\mBbbcj\mBacbc = 0,
\end{split}\\
\begin{split}
h_{16}(n) & = \Exx |B_{12}||B_{13}||B_{21}||B_{31}| = \Exx\mBabbj\mBacbj\mBba\mBcabj = -(n-3)\mu_3\Exx\mBabbd\mBabcd\mBbaid\mBbcad = 0,
\end{split}\\
\begin{split}
h_{17}(n) & = \Exx |B_{11}||B_{12}||B_{23}||B_{24}| = \Exx\mBaaib\mBab\mBbcib\mBbdib = -(n-2)\mu_3\Exx\mBacab\mBabcj\mBbcbc\mBbcbd = 0,
\end{split}\\
\begin{split}
h_{18}(n) & = \Exx |B_{11}||B_{12}||B_{23}||B_{34}| = \Exx\mBaa\mBabia\mBbcia\mBcdia = (n-3)\mu_3 \Exx\mBaadj\mBadab\mBbdac\mBcdad = 0,
\end{split}\\
\begin{split}
h_{19}(n) & = \Exx|B_{11}||B_{12}||B_{13}||B_{14}| = \Exx\mBaa\mBabia\mBacia\mBadia = (n-1)\mu_3 \Exx\mBaabj\mBabab \mBabac\mBabad = 0,
\end{split}\\
\begin{split}
h_{20}(n) & = \Exx|B_{11}||B_{22}||B_{33}||B_{44}| = \Exx\mBaa\mBbbaj\mBccaj\mBddaj =\\
    & = \mu_3 \Exx\mBaa\mBabab\mBacac\mBadad - (n-4)\mu_3\Exx\mBaaie\mBabbe\mBacce\mBadde =\\
    & = \mu_3 \left[\Exx|B||B_{11}||B_{22}||B_{33}|\right]_{n\rightarrow n-1} = \mu_3 h_6(n-1) = \mu_3^4 g_4(n-4).
\end{split}
\end{align}
\end{proof}

\subsection{Conclusion}
\begin{definition}[$H_0(t),H_9(t),H_{10}(t)$]
Given $h_0(n),h_9(n),h_{10}(n)$ as before, we define auxiliary generating functions
\begin{equation}
    H_0(t) = \sum_{n=0}^\infty \frac{t^n}{(n!)^2} n^2 h_0(n), \quad
    H_9(t) = \sum_{n=0}^\infty \frac{t^n}{(n!)^2} n^2 (n-1)^2 h_9(n), \quad
    H_{10}(t) = \sum_{n=0}^\infty \frac{t^n}{(n!)^2} n^2 (n-1)^2 h_{10}(n). 
\end{equation}
\end{definition}
\begin{proposition}\label{PropH}
\begin{equation}
    H_0(t) = \frac{\mu_2 t G_4(t)}{1-\mu_2^2 t},\qquad\qquad
    H_9(t) = \mu_2^2 t^2 \frac{(1+\mu_2^2t)G_4(t)}{(1-\mu_2^2 t)^2}, \qquad\qquad
    H_{10}(t) = \frac{\mu_2^4 t^3 G_4(t)}{(1-\mu_2^2t)^2}.
\end{equation}
\end{proposition}
\begin{proof}
By summing up the recurrence relations for $h_0(n), h_9(n), h_{10}(n)$ in Propositions \ref{Lem3} and \ref{Lem5}, we get
\begin{equation}
\begin{split}
    H_0(t) & = \mu_2 t G_4(t) + \mu_2^2 t H_0(t), \\
    H_9(t) & = \mu_2 t H_0(t) + \mu_2^3 t^2 H_0(t) + \mu_2^2 t H_9(t), \\
    H_{10}(t) & = \mu_2^3 t^2 H_0(t) + \mu_2^2 t H_{10}(t).
\end{split}
\end{equation}
By using simple algebraic manipulations, we get the desired statement.
\end{proof}
\begin{corollary}
By using Corollary \ref{ColSummands} and Propositions \ref{Lem2}, \ref{Lem3}, \ref{Lem4}, \ref{Lem5} and \ref{PropH}, we get, by summation,
\begin{equation}
\begin{split}
    F_4(t) = G_4(t) & + 4m_1 \mu_3 t G_4(t) + 6m_1^2 (H_0(t) + \mu_3^2 t^2 G_4(t)) + 4 m_1^3 (3\mu_3 t H_0(t) + \mu_3^3 t^3 G_4(t)) \\
& + m_1^4 (t G_4(t) + 6 \mu_2 t H_0(t) + 3 H_9(t) + 6 H_{10}(t) + 6 \mu_3^2 t^2 H_0(t) + \mu_3^4 t^4 G_4(t)),
\end{split}
\end{equation}
from which Theorem \ref{MainThm} follows immediately.
\end{corollary}

\clearpage
\section{Proof of Theorem \ref{MainThmDembo}}
\subsection{Generating function expansion}
We can expand $F_4(t,\omega)$ in a form of a Taylor-like series in $t$ and $\omega$,
\begin{equation}\label{F4expaEq}
F_4(t,\omega) = \sum_{j=0}^\infty \omega^j \left(\frac{1-\mu_2^2 t}{1-\omega-\mu_2^2 t}\right)^{j+1} \Phi_j(t).
\end{equation}
We claim this is without loss of generality. To see this, perform the Taylor expansion of the bracket, write $\Phi_j(t)$ as a series in $t$ and then compare the $t$ and $\omega$ coefficients with $f_4(n,p)$ in \eqref{Fkto}. This particular choice of expansion was made to make the \textbf{binomial transform} in $\omega$ behave nicely, we have
\begin{equation}
    \btrans{F_4}(t,\omega) = \frac{1}{1-\omega} F_4\left(t,\frac{\omega \mu_2 ^2 t}{1-\omega}\right) = \sum_{j=0}^\infty \omega^j \left(\frac{1-\mu_2^2 t}{1-\omega-\mu_2^2 t}\right)^{j+1} \mu_2^{2j} t^{j} \Phi_j(t).
\end{equation}
The proof of Theorem \ref{MainThmDembo} relies on a crucial fact that
\begin{equation}\label{Phi3}
    \Phi_j(t) =0 \qquad \text{for} \qquad j\geq 3.
\end{equation}
That is,
\begin{equation}
    F_4(t,\omega) = \frac{1-\mu_2^2 t}{1-\omega-\mu_2^2 t} \left( \Phi_0(t) + \omega \frac{1-\mu_2^2 t}{1-\omega-\mu_2^2 t} \Phi_1(t) + \omega^2 \left(\frac{1-\mu_2^2 t}{1-\omega-\mu_2^2 t}\right)^2 \Phi_2(t)\right).
\end{equation}
The remaining functions $\Phi_0(t),\Phi_1(t),\Phi_2(t)$ can be then found just by the methods used proving \ref{MainThm}. Specially,
\begin{equation}\label{SpecPhi0}
    \Phi_0(t) = F_4(t,0) = F_4(t) = \frac{e^{t(\mu_4 - 3 \mu_2^2)}}{(1-\mu_2^2t)^5}\left(1+\sum_{k=1}^6 p_k t^k\right).
\end{equation}
To show \eqref{Phi3}, we use \textbf{Cauchy-Binet formula}.

\subsection{Cauchy-Binet formula}
\begin{proposition}[Cauchy-Binet formula]
Let $C = (c_{ij})_{n \times p}$ and $D = (d_{ij})_{n \times p}$ be real matrices and $C_{(i_1,i_2,\ldots,i_p)}$ and $D_{(i_1,i_2,\ldots,i_p)}$ be square matrices formed from those by selecting the rows $i_1,i_2,\ldots i_p$, then
\begin{equation}
    |C^T D| = \sum_{1\leq i_1 < i_2 < \ldots < i_p \leq n} |C_{(i_1,i_2,\ldots,i_p)}||D_{(i_1,i_2,\ldots,i_p)}|.
\end{equation}
Note that there is an equivalent formulation using \textbf{deleting} rows instead of selecting. Namely, denoting $C_{[ i_1,i_2,\ldots,i_{p'}]}$ a matrix formed from $C$ by deleting its rows $i_1,\ldots i_{p'}$, we have then
\begin{equation}
    |C^T D| = \sum_{1\leq i_1 < i_2 < \ldots < i_{n-p} \leq n} |C_{[i_1,i_2,\ldots,i_{n-p}]}||D_{[i_1,i_2,\ldots,i_{n-p}]}|.
\end{equation}\end{proposition}
\begin{remark}
If $p>n$, $|C^TD| = 0$ automatically. That means also that $f_k(n,p) = 0$ whenever $p>n$.
\end{remark}
\noindent
As stated earlier, the formula offers a simple derivation of $f_2(n,p)$. This is done by choice $C=D=U$, so
\begin{equation}\label{CB_UTU}
 |U^T U| = \sum_{1\leq i_1 < i_2 < \ldots < i_p \leq n} |U_{(i_1,i_2,\ldots,i_p)}|^2.
\end{equation}
Then, taking the expectation and by linearity, we get $ \binom{n}{p}$ identical terms, each attending the value \cite{fortet1951random}
\begin{equation}
    \Exx |U_{(1,2,\ldots,p)}|^2 = f_2(p) = p! (m_2 + m_1^2(p-1)) (m_2 - m_1^2)^{p-1},
\end{equation}
hence
\begin{equation}
    \Exx |U^T U| = p! \binom{n}{p} (m_2 + m_1^2(p-1)) (m_2 - m_1^2)^{p-1}.
\end{equation}
Somewhat similarly, to derive $\Exx |U^T U|^2$, we just square \eqref{CB_UTU} and take the expectation repeatedly.

\subsection{Dembo's generating function}
To illustrate the squaring technique, we rederive Dembo's formula for $F^{\mathrm{sym}}_4(t,\omega)$. First, squaring \eqref{CB_UTU},
\begin{equation}\label{CBsquared}
    |U^T U|^2 = \sum_{\substack{1\leq i_1 < i_2 < \ldots < i_p \leq n \\ 1\leq i'_1 < i'_2 < \ldots < i'_p \leq n}} |U_{(i_1,i_2,\ldots,i_p)}|^2\, |U_{(i'_1,i'_2,\ldots,i'_p)}|^2.
\end{equation}

\begin{definition}[$c_{p,q}$]
Given two identical copies of the set $\{ 1,2,3,\ldots,n\} $, we denote $c_{p,q}$ the number of ways how we can select $p$ numbers from the first copy and other $p$ numbers from the second copy, provided that exactly $q$ numbers in both selections were chosen simultaneously. Using standard combinatorics,
\begin{equation}
    c_{p,q} = \binom{n}{q}\binom{n-q}{p-q}\binom{n-p}{p-q} = \frac{n!}{ q!(p-q)!^2(n-2p+q)!}.
\end{equation}
\end{definition}

\begin{definition}[$U^{\left[q\right]},\tilde{U}^{\left[q\right]}$]
Denote $U^{\left[q\right]} = (X_{ij})_{p \times p}$ and $\tilde{U}^{\left[q\right]} = (\tilde{X}_{ij})_{p \times p}$ a random pair of $p$ by $p$ square matrices being identical in the first $q$ columns, that is $X_{ij} = \tilde{X}_{ij}$ for $j\leq q$ and all $i$. Otherwise, in columns $j>q$, we assume $\tilde{X}_{ij}$ are independent from each other and from all $X_{ij}$'s, following the same distribution.
\end{definition}

\begin{definition}[$\Exc$]
We denote $\Exc$ the conditional expectation taken with respect only to the entries in the $j>q$ columns of a random matrix pair. By properties of conditional expectations, $\Ex = \Exx \Exc$.
\end{definition}

Taking expectation of \eqref{CBsquared}, transposing each matrix and collecting identical terms, we get
\begin{equation}\label{f_np_pre}
    f_4(n,p) = \Exx |U^T U|^2 = \sum_{q=0}^p c_{p,q} \Exx |U^{\left[q\right]}|^2\, |\tilde{U}^{\left[q\right]}|^2.
\end{equation}
Now, we use the key assumption that $X_{ij}$'s follow a symmetrical distribution, that is $m_1=m_3=0$ and $f_4(n,p) = f^{\mathrm{sym}}_4(n,p)$. Expanding the independent columns of $U^{\left[q\right]}$ (and $\tilde{U}^{\left[q\right]}$, respectively) and taking $\Exc$,
\begin{equation}\label{ExxcUq}
    \Exxc |U^{\left[q\right]}|^2 = \Exxc |\tilde{U}^{\left[q\right]}|^2 = (p-q)!\,  m_2^{p-q} \!\!\!\!\!\!\!\! \sum_{1\leq i_1 < i_2 < \ldots < i_q \leq p} |U^{\left[q\right]}_{(i_1,i_2,\ldots,i_q)}|^2 = (p-q)! \, m_2^{p-q} |U'^T U'|,
\end{equation}
where in the last step we used Cauchy-Binet formula again and denoted $U'$ a $p \times q$ matrix formed from $U^{\left[q\right]}$ by selecting its first $q$ columns. Therefore
\begin{equation}
    \Exx |U^{\left[q\right]}|^2\, |\tilde{U}^{\left[q\right]}|^2 = \Exx \left[ \Exxc |U^{\left[q\right]}|^2\, \Exxc |\tilde{U}^{\left[q\right]}|^2 \right] = (p-q)!^2 m_2^{2(p-q)} \Exx |U'^T U'|^2 = (p-q)!^2 m_2^{2(p-q)} f_4(p,q).
\end{equation}
Inserting the result into \eqref{f_np_pre}, we get the recurrence relation
\begin{equation}\label{f_np}
    f^{\mathrm{sym}}_4(n,p) = \sum_{q=0}^p \frac{n! m_2^{2(p-q)} f^{\mathrm{sym}}_4(p,q)}{ q!(n-2p+q)!}.
\end{equation}
This is, inserting to \eqref{Fkto} and by straightforward manipulations, equivalent to
\begin{equation}
F^{\mathrm{sym}}_4(t,\omega) = \frac{1}{1-\omega} F^{\mathrm{sym}}_4\left(t,\frac{\omega m_2 ^2 t}{1-\omega}\right) = \btrans{F^{\mathrm{sym}}_4}(t,\omega).
\end{equation}
Using our ansatz for generating functions, namely
\begin{equation}
F^{\mathrm{sym}}_4(t,\omega) = \sum_{j=0}^\infty \omega^j \left(\frac{1-m_2^2 t}{1-\omega-m_2^2 t}\right)^{j+1} \Phi^{\mathrm{sym}}_j(t),
\end{equation}
we immediately obtain the condition
\begin{equation}
    \Phi^{\mathrm{sym}}_j(t) = m_2^{2j} t^{j} \Phi^{\mathrm{sym}}_j(t),
\end{equation}
from which
\begin{equation}
    \Phi^{\mathrm{sym}}_j(t) = 0 \qquad \text{for} \qquad j\geq 1.
\end{equation}
Therefore, by \eqref{f0F0},
\begin{equation}
F^{\mathrm{sym}}_4(t,\omega) = \frac{1-m_2^2 t}{1\!-\!\omega\!-\!m_2^2 t} \Phi^{\mathrm{sym}}_0(t) = \frac{1-m_2^2 t}{1\!-\!\omega\!-\!m_2^2 t} F^{\mathrm{sym}}_4(t,0) = \frac{1-m_2^2 t}{1\!-\!\omega\!-\!m_2^2 t} F^{\mathrm{sym}}_4(t) = \frac{e^{t(m_4 - 3m_2^2)}}{(1\!-\!m_2^2 t)^2(1\!-\!\omega\!-\!m_2^2 t)}.
\end{equation}

\subsection{Matrix resolvents}
\begin{definition}[$V,V'$]
Similarly as for $U$ and $U'$, which are given as
\begin{equation}
U =
    \begin{pmatrix} 
        X_{11} & \dots  & X_{1p}\\
        \vdots & \ddots & \vdots\\
        X_{n1} & \dots  & X_{np}
    \end{pmatrix} \qquad \text{and} \qquad 
U' =
    \begin{pmatrix} 
        X_{11} & \dots  & X_{1q}\\
        \vdots & \ddots & \vdots\\
        X_{p1} & \dots  & X_{pq}
    \end{pmatrix},
\end{equation}
we denote
\begin{equation}
V =
    \begin{pmatrix} 
        Y_{11} & \dots  & Y_{1p}\\
        \vdots & \ddots & \vdots\\
        Y_{n1} & \dots  & Y_{np}
    \end{pmatrix} \qquad \text{and} \qquad 
V' =
    \begin{pmatrix} 
        Y_{11} & \dots  & Y_{1q}\\
        \vdots & \ddots & \vdots\\
        Y_{p1} & \dots  & Y_{pq}
    \end{pmatrix}.
\end{equation}
\end{definition}

\begin{definition}[$\stkout{U},\stkout{U}',\stkout{V},\stkout{V}'$]
Denote $\stkout{U}$ an $n \times (p + 1)$ matrix formed from $U$ by attaching to it $(p\!+\!1)$-th column filled with $1$'s. Similarly, $\stkout{U}'$ be a $p\times (q+1)$ matrix formed from $U'$ by the same way. Symbolically,
\begin{equation}
\stkout{U} =
    \begin{pmatrix} 
        X_{11} & \dots  & X_{1p} & 1\\
        \vdots & \ddots & \vdots & \vdots\\
        X_{n1} & \dots  & X_{np} & 1
    \end{pmatrix} \qquad \text{and} \qquad 
\stkout{U}' =
    \begin{pmatrix} 
        X_{11} & \dots  & X_{1q} & 1\\
        \vdots & \ddots & \vdots & \vdots\\
        X_{p1} & \dots  & X_{pq} & 1
    \end{pmatrix}.
\end{equation}
Similarly, we denote
\begin{equation}
\stkout{V} =
    \begin{pmatrix} 
        Y_{11} & \dots  & Y_{1p} & 1\\
        \vdots & \ddots & \vdots & \vdots\\
        Y_{n1} & \dots  & Y_{np} & 1
    \end{pmatrix} \qquad \text{and} \qquad 
\stkout{V}' =
    \begin{pmatrix} 
        Y_{11} & \dots  & Y_{1q} & 1\\
        \vdots & \ddots & \vdots & \vdots\\
        Y_{p1} & \dots  & Y_{pq} & 1
    \end{pmatrix}.
\end{equation}
And finally, as a special case when $p=0$, we write
\begin{equation}
    \stkout{U} = \stkout{V} =     \left.\begin{pmatrix} 
        1\\
        \vdots\\
        1
    \end{pmatrix}\right\rbrace p \qquad \text{and} \qquad
    \stkout{U}' = \stkout{V}' =     \left.\begin{pmatrix} 
        1\\
        \vdots\\
        1
    \end{pmatrix}\right\rbrace q.
\end{equation}
\end{definition}

\begin{lemma}\label{UTUVTV}
\begin{equation}
    |\stkout{U}^T\stkout{U}| = |\stkout{V}^T\stkout{V}|
\end{equation}
\end{lemma}
\begin{proof}
Using Cauchy-Binet formula,
\begin{equation}
    |\stkout{U}^T \stkout{U}| = \sum_{1\leq i_1 < i_2 < \ldots < i_{p+1} \leq n} |\stkout{U}_{(i_1,i_2,\ldots,i_{p+1})}|^2.
\end{equation}
By properties of determinant, note that, assuming $m_1 \neq 0$,
\begin{equation}
    |\stkout{U}_{(i_1,i_2,\ldots,i_{p+1})}| =
    \begin{vmatrix} 
        X_{i_1 1} & \dots  & X_{i_1 p} & 1\\
        \vdots & \ddots & \vdots & \vdots\\
        X_{i_{p+1} 1} & \dots  & X_{i_{p+1} p} & 1
    \end{vmatrix} = \frac{1}{m_1}\begin{vmatrix} 
        X_{i_1 1} & \dots  & X_{i_1 p} & m_1\\
        \vdots & \ddots & \vdots & \vdots\\
        X_{i_{p+1} 1} & \dots  & X_{i_{p+1} p} & m_1
    \end{vmatrix}.
\end{equation}
By Lemma \ref{MainLem}, choosing $\lambda = m_1$, $u = (\underbrace{1,\ldots,1}_{p+1})^T$ and $v = (\underbrace{1,\ldots,1}_{p+1})^T$,
\begin{equation}
    \begin{vmatrix} 
        X_{i_1 1} & \dots  & X_{i_1 p} & m_1\\
        \vdots & \ddots & \vdots & \vdots\\
        X_{i_{p+1} 1} & \dots  & X_{i_{p+1} p} & m_1
    \end{vmatrix} = \begin{vmatrix} 
        Y_{i_1 1} & \dots  & Y_{i_1 p} & 0\\
        \vdots & \ddots & \vdots & \vdots\\
        Y_{i_{p+1} 1} & \dots  & Y_{i_{p+1} p} & 0
    \end{vmatrix} + m_1\sum_{ij} (-1)^{i+j}\det \begin{pmatrix} 
        Y_{i_1 1} & \dots & Y_{i_1 p} & 0\\
        \vdots & \ddots & \vdots & \vdots\\
        Y_{i_{p+1} 1} & \dots  & Y_{i_{p+1} p} & 0
    \end{pmatrix}_{ij}.
\end{equation}
The first $Y$ determinant is automatically zero due to the last column filled with zeroes. Similarly, the only nonzero terms in the sum are those for $j=p+1$, thus
\begin{equation}
    |\stkout{U}_{(i_1,i_2,\ldots,i_{p+1})}| = \sum_{i} (-1)^{i+p+1}\det \begin{pmatrix} 
        Y_{i_1 1} & \dots & Y_{i_1 p} & 0\\
        \vdots & \ddots & \vdots & \vdots\\
        Y_{i_{p+1} 1} & \dots  & Y_{i_{p+1} p} & 0
    \end{pmatrix}_{i,p+1} = |\stkout{V}_{(i_1,i_2,\ldots,i_{p+1})}|,
\end{equation}
which we have identified as the expansion of $|\stkout{V}_{(i_1,i_2,\ldots,i_{p+1})}|$ in the last column. The proof of the proposition is finished by again employing the Cauchy-Binet formula. By continuity, the lemma holds even for $m_1 =0$
\end{proof}

\begin{definition}[$O^{\left[q\right]},V^{\left[q\right]}$]
Denote
\begin{equation*}
O^{\left[ q\right]} =
    \begin{pmatrix} 
        X_{11} & \dots  & X_{1q} & Y_{1,q+1} & \dots  & Y_{1,p}\\
        \vdots & \ddots & \vdots & \vdots & \ddots  & \vdots\\
        X_{p1} & \dots  & X_{pq} & Y_{p,q+1} & \dots  & Y_{p,p}\\
    \end{pmatrix} \qquad \text{and} \qquad
V^{\left[ q\right]} =
    \begin{pmatrix} 
        Y_{11} & \dots  & Y_{1q} & Y_{1,q+1} & \dots  & Y_{1,p}\\
        \vdots & \ddots & \vdots & \vdots & \ddots  & \vdots\\
        Y_{p1} & \dots  & Y_{pq} & Y_{p,q+1} & \dots  & Y_{p,p}\\
    \end{pmatrix}.
\end{equation*}
\end{definition}

\begin{proposition}\label{PropOqgen}
\begin{equation*}
    \Exxc |O^{\left[q\right]}|^2 = (p-q)! \, \mu_2^{p-q} |U'^T U'|, \qquad \qquad \Exxc |V^{\left[q\right]}|^2 = (p-q)! \, \mu_2^{p-q} |V'^T V'|.
\end{equation*}
\end{proposition}
\begin{proof}
Thanks to $Y_{ij}$'s being central, the formulae are a direct consequence of \eqref{ExxcUq}.
\end{proof}

\begin{proposition}\label{PropVq}
\begin{equation*}
    |U^{\left[q\right]}| = |O^{\left[q\right]}| + m_1 \sum_{\substack{i \in \{1,\ldots,p \} \\ j \in \{q+1,\ldots,p \} }} (-1)^{i+j}|O^{\left[q\right]}_{ij}|.
\end{equation*}
\end{proposition}
\begin{proof}
Use Lemma \ref{MainLem} with $\lambda = m_1, C = O^{\left[q\right]},u = (\underbrace{1,\ldots,1}_{p})^T$ and $v = (\underbrace{0,\ldots,0}_{q},\underbrace{1,\ldots,1}_{p-q})^T.$
\end{proof}

\begin{proposition}\label{PropUqgen}
When $X_{ij}$'s follow general distribution, we have, in contrast to \eqref{ExxcUq},
\begin{equation*}
    \Exxc |U^{\left[q\right]}|^2 = (p-q)! \, \mu_2^{p-q} \left(|U'^T U'| + \frac{m_1^2}{\mu_2}|\stkout{V}'^T\stkout{V}'|\right)
\end{equation*}
\end{proposition}
\begin{proof}
Squaring the Proposition \ref{PropVq},
\begin{equation}
    |U^{\left[q\right]}|^2 = |O^{\left[q\right]}|^2 + 2m_1\!\!\!\!\! \sum_{\substack{i \in \{1,\ldots,p \} \\ j \in \{q+1,\ldots,p \} }} (-1)^{i+j}|O^{\left[q\right]}||O^{\left[q\right]}_{ij}|+ m_1^2\!\!\!\!\!\sum_{\substack{i,r \in \{1,\ldots,p \} \\ j,s \in \{q+1,\ldots,p \} }} (-1)^{i+j+r+s}|O^{\left[q\right]}_{ij}||O^{\left[q\right]}_{rs}|.
\end{equation}
We now take the $\Exc$ expectation. Expanding $|O^{\left[q\right]}|$ in the $j$-th column, where $j>q$, we notice that
\begin{equation}
    \Exxc |O^{\left[q\right]}||O^{\left[q\right]}_{ij}| = \Exxc \sum_{k=1}^{p} (-1)^{k+j}Y_{kj}|O^{\left[q\right]}_{kj}||O^{\left[q\right]}_{ij}| = 0
\end{equation}
since $\Exxc Y_{kj} = 0$. Similarly, by expanding $|O^{\left[q\right]}_{ij}|$ in the $s$-th column, where $s>q$, we get
\begin{equation}
    \Exxc |O^{\left[q\right]}_{ij}||O^{\left[q\right]}_{rs}| = 0 \qquad \text{when} \qquad s \neq j.
\end{equation}
Only terms with $j=s$ survive. Therefore, by symmetry,
\begin{equation}\label{ExUq}
\Exxc |U^{\left[q\right]}|^2 = \Exxc |O^{\left[q\right]}|^2 + m_1^2 (p-q) \, \Exxc \bigg{(}\sum_{i = 1}^n (-1)^{i+j}|O^{\left[q\right]}_{i,q+1}| \bigg{)}^2.
\end{equation}
Notice that in \eqref{ExUq}, we can interpret the sum in the bracket as another expansion, namely
\begin{equation}
    \sum_{i=1}^p (-1)^{i+j}|O^{\left[q\right]}_{i,q+1}| =
    \begin{vmatrix} 
        X_{11} & \dots  & X_{1q} & 1 & Y_{1,q+2} & \dots & Y_{1p} \\
        \vdots &  \ddots & \vdots & \vdots & \vdots & \ddots & \vdots\\
        X_{p1} & \dots  & X_{pq} & 1 & Y_{1,q+2} & \dots & Y_{1p}
    \end{vmatrix}.
\end{equation}
Thus, taking $\Exc$ (that is, taking the expectation with respect to $Y_{ij}$'s), we get
\begin{equation}
\Exxc \bigg{(}\sum_{i=1}^p (-1)^{i+j}|O^{\left[q\right]}_{i,q+1}| \bigg{)}^2 = (p-q-1)! \, \mu_2^{p-q-1} |\stkout{U}'^T \stkout{U}'|.
\end{equation}
Together with Lemma \ref{UTUVTV}, we have $|\stkout{U}'^T \stkout{U}'| = |\stkout{V}'^T \stkout{V}'|$, which concludes the proof of the proposition.
\end{proof}

\begin{definition}[$\stkout{V}^{\left[q\right]},\tilde{\stkout{V}}^{\left[ q\right]},\Exxc$]
Denote
\begin{equation}
\stkout{V}^{\left[ q\right]} =
    \begin{pmatrix} 
        Y_{11} & \dots  & Y_{1q} & Y_{1,q+1} & \dots  & Y_{1,p-1} & 1\\
        \vdots & \ddots & \vdots & \vdots & \ddots  & \vdots & \vdots\\
        Y_{p1} & \dots  & Y_{pq} & Y_{p,q+1} & \dots  & Y_{p,p-1} & 1\\
    \end{pmatrix}, \,\,\,\,\, 
\tilde{\stkout{V}}^{\left[ q\right]} =
    \begin{pmatrix} 
        Y_{11} & \dots  & Y_{1q} & \tilde{Y}_{1,q+1} & \dots  & \tilde{Y}_{1,p-1} & 1\\
        \vdots & \ddots & \vdots & \vdots & \ddots  & \vdots & \vdots\\
        Y_{p1} & \dots  & Y_{pq} & \tilde{Y}_{p,q+1} & \dots  & \tilde{Y}_{p,p-1} & 1 \\
    \end{pmatrix}
\end{equation}
a random pair of $p$ by $p$ square matrices being identical in the first $q$ columns, that is $Y_{ij} = \tilde{Y}_{ij}$ for $j\leq q$ and all $i$. Otherwise, in columns $j>q$, we assume $\tilde{Y}_{ij}$ are independent from each other and from all $Y_{ij}$'s, following the same distribution. For a reminder, the $\Exc$ expectation is constructed such it acts only on those $j>q$ columns.
\end{definition}

\begin{proposition}\label{PropVqgen}
\begin{equation*}
    \Exxc |\stkout{V}^{\left[ q\right]}|^2 = (p-q-1)! \, \mu_2^{p-q-1} |\stkout{V}'^T \stkout{V}'|
\end{equation*}
\end{proposition}
\begin{proof}
Denoting
\begin{equation}
    \dout{V}' = \begin{pmatrix} 
        Y_{11} & \dots  & Y_{1q} & 1 & 1\\
        \vdots & \ddots & \vdots & \vdots & \vdots\\
        Y_{p1} & \dots  & Y_{pq} & 1 & 1\\
    \end{pmatrix}
\end{equation}
we have, shifting the $1$'s column to the left and by using Proposition \ref{PropUqgen},
\begin{equation}
    \Exxc |\stkout{V}^{\left[ q\right]}|^2 =
    \Exxc\begin{vmatrix} 
        Y_{11} & \dots  & Y_{1q} & 1 & Y_{1,q+1} & \dots  & Y_{1,p-1}\\
        \vdots & \ddots & \vdots & \vdots & \vdots & \ddots  & \vdots\\
        Y_{p1} & \dots  & Y_{pq} & 1 & Y_{p,q+1} & \dots  & Y_{p,p-1}\\
    \end{vmatrix}^2 = (p-1-q)! \, \mu_2^{p-1-q} \left(|\stkout{V}'^T \stkout{V}'| + \frac{m_1^2}{\mu_2}|\dout{V}'^T\dout{V}'|\right).
\end{equation}
However, $|\dout{V}'^T\dout{V}'|=0$ trivially (parallelepiped with two spanning vectors being identical has zero volume).
\end{proof}

\begin{definition}[$W,W',\sigma,\sigma'$]
Let us define another pair of random matrices
\begin{equation}
W =
    \begin{pmatrix} 
        Y_{11} & \dots  & Y_{1p}\\
        \vdots & \ddots & \vdots\\
        Y_{n1} & \dots  & Y_{np} \\
        1 & \dots  & 1 \\
    \end{pmatrix} \qquad \text{and} \qquad 
W' =
    \begin{pmatrix} 
        Y_{11} & \dots  & Y_{1q}\\
        \vdots & \ddots & \vdots\\
        Y_{p1} & \dots  & Y_{pq} \\
        1 & \dots  & 1 \\
    \end{pmatrix},
\end{equation}
from which we construct two sums $\sigma$ and $\sigma'$ as
\begin{equation}
    \sigma = \sum_{1\leq i_1 < i_2 < \ldots < i_{p-1} \leq n} |W_{(i_1,i_2,\ldots,i_{p-1})}|^2 \qquad \text{and} \qquad \sigma' = \sum_{1\leq i_1 < i_2 < \ldots < i_{q-1} \leq n} |W'_{(i_1,i_2,\ldots,i_{q-1})}|^2.
\end{equation}
By definition, we put $\sigma = 0$ when $p=0$.
\end{definition}

\begin{definition}[$W^{\left[ q\right]},\tilde{W}^{\left[ q\right]}$]
Denote
\begin{equation*}
W^{\left[ q\right]} =
    \begin{pmatrix} 
        Y_{11} & \dots  & Y_{1q} & Y_{1,q+1} & \dots  & Y_{1,p+1}\\
        \vdots & \ddots & \vdots & \vdots & \ddots  & \vdots\\
        Y_{p1} & \dots  & Y_{pq} & Y_{p,q+1} & \dots  & Y_{p,p+1} \\
        1 & \dots  & 1 & 1 & \dots  & 1 \\
    \end{pmatrix}, \qquad 
\tilde{W}^{\left[ q\right]} =
    \begin{pmatrix} 
        Y_{11} & \dots  & Y_{1q} & \tilde{Y}_{1,q+1} & \dots  & \tilde{Y}_{1,p+1}\\
        \vdots & \ddots & \vdots & \vdots & \ddots  & \vdots\\
        Y_{p1} & \dots  & Y_{pq} & \tilde{Y}_{p,q+1} & \dots  & \tilde{Y}_{p,p+1} \\
        1 & \dots  & 1 & 1 & \dots  & 1 \\
    \end{pmatrix}.
\end{equation*}
\end{definition}

\begin{proposition}\label{PropWqgen}
\begin{equation*}
    \Exxc |W^{\left[ q\right]}|^2 = (p-q+1)! \, \mu_2^{p-q} \left(\mu_2 \sigma' + |V'^T V'|\right)
\end{equation*}
\end{proposition}
\begin{proof}
Denote
\begin{equation}
Z^{\left[ q\right]} =
    \begin{pmatrix} 
    Y_{11} & \dots  & Y_{1q} & Y_{1,q+1} & \dots  & Y_{1,p+1}\\
    \vdots & \ddots & \vdots & \vdots & \ddots  & \vdots\\
    Y_{p1} & \dots  & Y_{pq} & Y_{p,q+1} & \dots  & Y_{p,p+1} \\
    1 & \dots  & 1 & 0 & \dots  & 0 \\
    \end{pmatrix}.
    \end{equation}
By Lemma \ref{MainLem} (or by expansion in the last row), we get
\begin{equation}
    |W^{\left[q\right]}| = |Z^{\left[q\right]}| + \sum_{j=q+1}^{p+1} (-1)^{p+1+j}|Z^{\left[q\right] }_{p+1,j}|.
\end{equation}
Squaring,
\begin{equation}
    |W^{\left[q\right]}|^2 = |Z^{\left[q\right]}|^2 + 2\sum_{j=q+1}^{p+1} (-1)^{p+1+j}|Z^{\left[q\right]}||Z^{\left[q\right] }_{p+1,j}|+\sum_{j=q+1}^{p+1}\sum_{s=q+1}^{p+1} (-1)^{j+s}|Z^{\left[q\right] }_{p+1,j}||Z^{\left[q\right] }_{p+1,s}|.
\end{equation}
We now take the $\Exc$ expectation. Expanding $|Z^{\left[q\right]}|$ in the $j$-th column, where $j>q$, we notice that
\begin{equation}
    \Exxc |Z^{\left[q\right]}||Z^{\left[q\right]}_{p+1,j}| = \Exxc \sum_{k=1}^{p} (-1)^{k+j}Y_{kj}|Z^{\left[q\right]}_{kj}||Z^{\left[q\right]}_{p+1,j}| = 0
\end{equation}
since $\Exxc Y_{kj} = 0$. Similarly, expanding $|Z^{\left[q\right]}_{p+1,j}|$ in the $s$-th column, where $s>q$, we get
\begin{equation}
    \Exxc |Z^{\left[q\right] }_{p+1,j}||Z^{\left[q\right] }_{p+1,s}| = 0 \qquad \text{when} \qquad s \neq j.
\end{equation}
Only terms with $j=s$ survive. Therefore, by symmetry,
\begin{equation}
\Exxc |W^{\left[q\right]}|^2 = \Exxc |Z^{\left[q\right]}|^2 + (p-q+1) \, \Exxc  |Z^{\left[q\right]}_{p+1,p+1}|^2.
\end{equation}
But $Z^{\left[q\right]}_{p+1,q+1} = V^{\left[q\right]}$. So, using Proposition \ref{PropOqgen},
\begin{equation}
    \Exxc |Z^{\left[q\right]}_{p+1,q+1}|^2 = (p-q)! \, \mu_2^{p-q} |V'^T V'|.
\end{equation}
On the other hand, expanding $|Z^{\left[q\right]}|$ in all columns with $j>q$ and collecting terms with the same value,
\begin{equation}
    \Exxc |Z^{\left[q\right]}|^2 = (p-q+1)! \, \mu_2^{p-q+1} \sigma'.
\end{equation}
\end{proof}

\subsection{Auxiliary moments and their recurrences}
\begin{definition}[$B,g_4(n),g_4(n,p),G_4(t),G_4(t,\omega),\Psi_j(t)$] Given $Y_{ij}$'s, we form a matrix $B=(Y_{ij})_{n \times n}$ and denote $g_4(n) = \Exx |B|^4$ and
\begin{equation}
    G_4(t) = \sum_{n=0}^\infty \frac{t^n}{(n!)^2} g_4(n).
\end{equation}
Similarly, denote $g_4(n,p) = \Exx |V^T V|^2$. For its generating function, we write
\begin{equation}
        G_4(t, \omega) = \sum_{n=0}^\infty \sum_{p=0}^n\frac{(n-p)! }{n!p!} t^p \omega^{n-p} g_4(n,p).
\end{equation}
\end{definition}
\begin{remark}
Since the moments of a random determinant are dependent only on moments of its random entries, we get that $g_4(n)$ and $g_4(n,p)$ are equal to  $f_4^{\mathrm{cen}}(n)$ and $f_4^{\mathrm{cen}}(n,p)$, respectively, in which we replace $m_j$ by $\mu_j$ when $j>1$. So, by Dembo's formula and by the ansatz expansion
\begin{equation}
    G_4(t,\omega) = \sum_{j=0}^\infty \omega^j \left(\frac{1-m_2^2 t}{1-\omega-m_2^2 t}\right)^{j+1} \Psi_j(t) = \frac{e^{t(\mu_4 - 3\mu_2^2)}}{(1-\mu_2^2 t)^2(1-\omega-\mu_2^2 t)},
\end{equation}
where
\begin{equation}
    \Psi_0(t) = G_4(t,0) = G_4(t) = \frac{e^{t(\mu_4-3\mu_2^2)}}{(1-\mu_2^2t)^3} \qquad \text{and} \qquad \Psi_j(t) = 0 \quad \text{for} \quad j\geq 1.
\end{equation}
\end{remark}

\begin{definition}[$\alpha(n,p),\beta(n,p),\gamma(n,p),\delta(n,p),\epsilon(n,p),\rho(n,p),\eta(n,p),\kappa(n,p)$] We define the following
\begin{align*}
        & \alpha(n,p) = \Exx |U^TU||\stkout{V}^T\stkout{V}|, && \beta(n,p) = |\stkout{V}^T\stkout{V}|^2, && \gamma(n,p) = \Exx |U^TU|\sigma, && \delta(n,p) = \Exx |U^TU||V^TV|,\\
        & \eta(n,p) = \Exx |V^TV||\stkout{V}^T\stkout{V}|, && \epsilon(n,p) = \Exx |\stkout{V}^T\stkout{V}|\sigma, && \rho(n,p) = \Exx |V^TV|\sigma, && \kappa(n,p) = \Exx \sigma^2.
\end{align*}
\end{definition}
\begin{remark}\label{RemSpecVals}
By definition, we have $\gamma(n,0) = \epsilon(n,0) = \rho(n,0) = \kappa(n,0) = 0$, $\delta(n,0) = 1$, $\alpha(n,0)=\eta(n,0) = n$ and $\beta(n,0) = n^2$. And due to vanishment for large $p$, we also have $\alpha(n,p) = \beta(n,p) = \epsilon(n,p) = \eta(n,p) = 0$ for $p \geq n$ and $\gamma(n,p) = \delta(n,p) = \rho(n,p)=0$ for $p\geq n+1$. Less straightforwardly, it can be shown that $\sigma = 0$ when $p \geq n+2$, so $\kappa(n,p) = 0$ for $p \geq n+2$ (but we don't need it).
\end{remark}
\begin{remark}\label{RemShift}
When $m_1 =0$, note that $f_4(n,p) = \delta(n,p) = g_4(n,p)$, $\gamma(n,p)=\rho(n,p)$ and $\alpha(n,p) = \eta(n,p)$.
\end{remark}

\begin{proposition}\label{PropSeqf4}
\begin{equation*}
    f_4(n,p) = \sum_{q=0}^p \frac{n! \mu_2^{2(p-q)}}{ q!(n-2p+q)!}\left(f_4(p,q)+\frac{2m_1^2}{\mu_2}\alpha(p,q)+\frac{m_1^4}{\mu_2^2}\beta(p,q)\right).
\end{equation*}
\end{proposition}
\begin{proof}
By Cauchy-Binet formula, taking expectation, transposing and collecting identical terms, we get
\begin{equation}\label{f4_np_pre}
    f_4(n,p) = \Exx |U^T U|^2 = \Ex\!\! \sum_{\substack{1\leq i_1 < i_2 < \ldots < i_p \leq n \\ 1\leq i'_1 < i'_2 < \ldots < i'_p \leq n}} |U_{(i_1,i_2,\ldots,i_p)}|^2\, |U_{(i'_1,i'_2,\ldots,i'_p)}|^2 = \sum_{q=0}^p c_{p,q} \Exx |U^{\left[q\right]}|^2\, |\tilde{U}^{\left[q\right]}|^2.
\end{equation}
By Proposition \ref{PropUqgen} with the fact that $\Ex = \Exx\Exc$,
\begin{equation}
\begin{split}
        \Exx |U^{\left[q\right]}|^2\, |\tilde{U}^{\left[q\right]}|^2 & = \Exx \left((p-q)! \, \mu_2^{p-q} \left(|U'^T U'| + \frac{m_1^2}{\mu_2}|\stkout{V}'^T\stkout{V}'|\right) \right)^2 \\
    & = (p-q)!^2 \, \mu_2^{2(p-q)} \left(f_4(p,q)+\frac{2m_1^2}{\mu_2}\alpha(p,q)+\frac{m_1^4}{\mu_2^2}\beta(p,q)\right),
\end{split}
\end{equation}
inserting this result into \eqref{f4_np_pre}, we get the desired recurrence relation for $f_4(n,p)$.
\end{proof}

\begin{definition}[$d_{p,q}$]
Given two identical copies of the set $\{ 1,2,3,\ldots,n\} $, we denote $d_{p,q}$ the number of ways how we can select $p$ numbers from the first copy and $p+1$ numbers from the second copy, provided that exactly $q$ numbers in both selections were chosen simultaneously. Using standard combinatorics,
\begin{equation}
    d_{p,q} = \binom{n}{q}\binom{n-q}{p-q}\binom{n-p}{p+1-q} = \frac{n!}{ q!(p-q)!(p-q+1)!(n-2p+q-1)!}.
\end{equation}
\end{definition}

\begin{proposition}\label{PropSeqAlf}
\begin{equation*}
    \alpha(n,p) = \sum_{q=0}^p \frac{n! \mu_2^{2(p-q)}}{ q!(n-2p+q-1)!}\left(\mu_2\gamma(p,q)+\delta(p,q)+m_1^2\epsilon(p,q)+\frac{m_1^2}{\mu_2}\eta(p,q)\right).
\end{equation*}
\end{proposition}
\begin{proof}
By Cauchy-Binet formula and by taking expectation, transposing and collecting identical terms,
\begin{equation}\label{al_np_pre}
    \alpha(n,p) = \Exx|U^T U||\stkout{V}^T \stkout{V}| = \Ex \!\!\!\!\!\!\! \sum_{\substack{1\leq i_1 < i_2 < \ldots < i_p \leq n \\ 1\leq i'_1 < i'_2 < \ldots < i'_{p+1} \leq n}} |U_{(i_1,i_2,\ldots,i_p)}|^2\, |\stkout{V}_{(i'_1,i'_2,\ldots,i'_{p+1})}|^2 = \sum_{q=0}^p d_{p,q} \Exx |U^{\left[q\right]}|^2\, |\tilde{W}^{\left[q\right]}|^2.
\end{equation}
Combining Propositions \ref{PropUqgen} and \ref{PropWqgen} with the fact that $\Ex = \Exx\Exc$,
\begin{equation}
\begin{split}
        \Exx |U^{\left[q\right]}|^2\, |\tilde{W}^{\left[q\right]}|^2 & = \Exx (p-q)! \, \mu_2^{p-q} \left(|U'^T U'| + \frac{m_1^2}{\mu_2}|\stkout{V}'^T\stkout{V}'|\right) (p-q+1)! \, \mu_2^{p-q} \left(\mu_2 \sigma' + |V'^T V'|\right) \\
    & = (p-q)! (p-q+1)! \, \mu_2^{2(p-q)} \left(\mu_2\gamma(p,q)+\delta(p,q)+m_1^2\epsilon(p,q)+\frac{m_1^2}{\mu_2}\eta(p,q)\right),
\end{split}
\end{equation}
inserting this result into \eqref{al_np_pre}, we get the desired relation for $\alpha(n,p)$.
\end{proof}

\begin{proposition}
\begin{equation*}
    \eta(n,p) = \sum_{q=0}^{p} \frac{n! \mu_2^{2(p-q)}}{ q!(n-2p+q-1)!}\left(\mu_2\rho(p,q)+g_4(p,q)\right).
\end{equation*}
\end{proposition}
\begin{proof}
In Proposition \ref{PropSeqAlf}, put $m_1 = 0$ and use Remark \ref{RemShift}.
\end{proof}

\begin{proposition}\label{PropSeqGam}
\begin{equation*}
    \gamma(n,p) = \sum_{q=0}^{p-1} \frac{n! \mu_2^{2p-2q-1}}{ q!(n-2p+q+1)!}\left(\alpha(p,q)+\frac{m_1^2}{\mu_2}\beta(p,q)\right).
\end{equation*}
\end{proposition}
\begin{proof}
By Cauchy-Binet formula and the definition of $\sigma$, taking expectation, transposing and collecting identical terms,
\begin{equation}\label{gam_np_pre}
    \gamma(n,p) = \Exx|U^TU|\sigma = \Ex \!\!\!\!\!\!\! \sum_{\substack{1\leq i_1 < i_2 < \ldots < i_p \leq n \\ 1\leq i'_1 < i'_2 < \ldots < i'_{p-1} \leq n}} |U_{(i_1,i_2,\ldots,i_p)}|^2\, |W_{(i'_1,i'_2,\ldots,i'_{p-1})}|^2 = \sum_{q=0}^{p-1} d_{p-1,q} \, \Exx |U^{\left[q\right]}|^2\, |\tilde{\stkout{V}}^{\left[q\right]}|^2.
\end{equation}
By using Propositions \ref{PropUqgen} and \ref{PropVqgen} and $\Ex = \Exx\Exc$,
\begin{equation}
\begin{split}
        \Exx |U^{\left[q\right]}|^2\, |\tilde{\stkout{V}}^{\left[q\right]}|^2 & = \Exx (p-q)! \, \mu_2^{p-q} \left(|U'^T U'| + \frac{m_1^2}{\mu_2}|\stkout{V}'^T\stkout{V}'|\right) \, (p-q-1)! \, \mu_2^{p-q-1} |\stkout{V}'^T \stkout{V}'|\\
    & = (p-q)!(p-q-1)! \, \mu_2^{2p-2q-1} \left(\alpha(p,q)+\frac{m_1^2}{\mu_2}\beta(p,q)\right),
\end{split}
\end{equation}
inserting this result into \eqref{gam_np_pre}, we get the desired relation for $\gamma(n,p)$.
\end{proof}

\begin{proposition}
\begin{equation*}
    \rho(n,p) = \sum_{q=0}^{p-1} \frac{n! \mu_2^{2p-2q-1}\eta(p,q)}{q!(n-2p+q+1)!}.
\end{equation*}
\end{proposition}
\begin{proof}
In Proposition \ref{PropSeqGam}, put $m_1 = 0$ and use Remark \ref{RemShift}.
\end{proof}

\begin{proposition}
\begin{equation*}
    \beta(n,p) = \sum_{q=0}^{p+1} \frac{n! \mu_2^{2(p-q)}}{ q!(n-2p+q-2)!}\left(\mu_2^2\kappa(p,q) + 2\mu_2\rho(p,q)+g_4(p,q)\right).
\end{equation*}
\end{proposition}
\begin{proof}
By Cauchy-Binet formula and by taking expectation, transposing and collecting identical terms,
\begin{equation}\label{bet_np_pre}
    \beta(n,p) = \Exx|\stkout{V}^T \stkout{V}|^2 = \Ex \!\!\!\!\!\!\! \sum_{\substack{1\leq i_1 < i_2 < \ldots < i_{p+1} \leq n \\ 1\leq i'_1 < i'_2 < \ldots < i'_{p+1} \leq n}} |\stkout{V}_{(i_1,i_2,\ldots,i_{p+1})}|^2\, |\stkout{V}_{(i'_1,i'_2,\ldots,i'_{p+1})}|^2 = \sum_{q=0}^{p+1} c_{p+1,q} \, \Exx |W^{\left[q\right]}|^2\, |\tilde{W}^{\left[q\right]}|^2.
\end{equation}
By using Proposition \ref{PropWqgen} and $\Ex = \Exx\Exc$,
\begin{equation}
        \Ex |W^{\left[q\right]}|^2\, |\tilde{W}^{\left[q\right]}|^2 \!\! = \!\Ex \left((p\!-\!q\!+\!1)! \mu_2^{p-q}\! \left(\mu_2 \sigma' \!\!+\! |V'^T V'|\right) \right)^2 \!\!\! =\! (p\!-\!q\!+\!1)!^2 \mu_2^{2(p-q)} \!\left(\mu_2^2\kappa(p,q)\!+\!2\mu_2\rho(p,q)\!+\!g_4(p,q)\right)
\end{equation}
and inserting this result into \eqref{bet_np_pre}, we get the desired relation for $\beta(n,p)$.
\end{proof}

\begin{proposition}
\begin{equation*}
    \kappa(n,p) = \sum_{q=0}^{p-1} \frac{n! \mu_2^{2(p-q-1)}\beta(p,q)}{q!(n-2p+q+2)!}.
\end{equation*}
\end{proposition}
\begin{proof}
By the definition of $\sigma$, taking expectation, transposing and collecting identical terms,
\begin{equation}\label{kap_np_pre}
    \kappa(n,p) = \Exx\sigma^2 = \Ex \!\!\!\!\!\!\! \sum_{\substack{1\leq i_1 < i_2 < \ldots < i_{p-1} \leq n \\ 1\leq i'_1 < i'_2 < \ldots < i'_{p-1} \leq n}} |W_{(i_1,i_2,\ldots,i_{p-1})}|^2\, |W_{(i'_1,i'_2,\ldots,i'_{p-11})}|^2 = \sum_{q=0}^{p-1} c_{p-1,q} \, \Exx |\stkout{V}^{\left[q\right]}|^2\, |\tilde{\stkout{V}}^{\left[q\right]}|^2.
\end{equation}
By using Proposition \ref{PropVqgen} and $\Ex = \Exx\Exc$,
\begin{equation}
        \Exx |\stkout{V}^{\left[q\right]}|^2\, |\tilde{\stkout{V}}^{\left[q\right]}|^2 = \Exx \left((p-q-1)! \, \mu_2^{p-q-1} |\stkout{V}'^T \stkout{V}'|\right)^2 = (p-q-1)!^2 \, \mu_2^{2(p-q-1)} \beta(p,q),
\end{equation}
inserting this result into \eqref{kap_np_pre}, we get the desired relation for $\kappa(n,p)$.
\end{proof}

\begin{proposition}
\begin{equation*}
    \delta(n,p) = \sum_{q=0}^{p} \frac{n! \mu_2^{2(p-q)}}{ q!(n-2p+q)!}\left(\delta(p,q)+\frac{m_1^2}{\mu_2}\eta(p,q)\right).
\end{equation*}
\end{proposition}
\begin{proof}
By Cauchy-Binet formula and by taking expectation, transposing and collecting identical terms,
\begin{equation}\label{del_np_pre}
    \delta(n,p) = \Exx|U^T U||V^T V| = \Ex \!\!\!\!\!\!\! \sum_{\substack{1\leq i_1 < i_2 < \ldots < i_p \leq n \\ 1\leq i'_1 < i'_2 < \ldots < i'_p \leq n}} |U_{(i_1,i_2,\ldots,i_p)}|^2\, |V_{(i'_1,i'_2,\ldots,i'_p)}|^2 = \sum_{q=0}^{p} c_{p,q} \, \Exx |U^{\left[q\right]}|^2\, |\tilde{V}^{\left[q\right]}|^2.
\end{equation}
By using Propositions \ref{PropOqgen} and \ref{PropUqgen} and $\Ex = \Exx\Exc$,
\begin{equation}
\begin{split}
        \Exx |U^{\left[q\right]}|^2\, |\tilde{V}^{\left[q\right]}|^2 & = \Exx (p-q)! \, \mu_2^{p-q} \left(|U'^T U'| + \frac{m_1^2}{\mu_2}|\stkout{V}'^T\stkout{V}'|\right) (p-q)! \, \mu_2^{p-q} |V'^T V'| \\
    & = (p-q)!^2 \, \mu_2^{2(p-q)} \left(\delta(p,q)+\frac{m_1^2}{\mu_2}\eta(p,q)\right),
\end{split}
\end{equation}
inserting this result into \eqref{del_np_pre}, we get the desired relation for $\delta(n,p)$.
\end{proof}

\begin{definition}[$e_{p,q}$]
Given two identical copies of the set $\{ 1,2,3,\ldots,n\} $, we denote $e_{p,q}$ the number of ways how we can select $p+1$ numbers from the first copy and $p-1$ numbers from the second copy, provided that exactly $q$ numbers in both selections were chosen simultaneously. Using standard combinatorics,
\begin{equation}
    e_{p,q} = \binom{n}{q}\binom{n-q}{p+1-q}\binom{n-(p+1)}{p-1-q} = \frac{n!}{ q!(p-q+1)!(p-q-1)!(n-2p+q)!}.
\end{equation}
\end{definition}

\begin{proposition}\label{PropSeqEpsi}
\begin{equation*}
    \epsilon(n,p) = \sum_{q=0}^{p-1} \frac{n! \mu_2^{2p-2q-1}}{ q!(n-2p+q)!}\left(\mu_2\epsilon(p,q)+\eta(p,q)\right).
\end{equation*}
\end{proposition}
\begin{proof}
By Cauchy-Binet formula and the definition of $\sigma$, taking expectation, transposing and collecting identical terms,
\begin{equation}\label{thet_np_pre}
    \epsilon(n,p) = \Exx |\stkout{V}^T\stkout{V}|\sigma = \Ex \!\!\!\!\!\!\! \sum_{\substack{1\leq i_1 < i_2 < \ldots < i_{p+1} \leq n \\ 1\leq i'_1 < i'_2 < \ldots < i'_{p-1} \leq n}} |\stkout{V}_{(i_1,i_2,\ldots,i_{p+1})}|^2\, |W_{(i'_1,i'_2,\ldots,i'_{p-1})}|^2 = \sum_{q=0}^{p-1} e_{p,q} \, \Exx |W^{\left[q\right]}|^2\, |\tilde{\stkout{V}}^{\left[q\right]}|^2.
\end{equation}
By using Propositions \ref{PropVqgen} and \ref{PropWqgen} and $\Ex = \Exx\Exc$,
\begin{equation}
\begin{split}
        \Exx |W^{\left[q\right]}|^2\, |\tilde{\stkout{V}}^{\left[q\right]}|^2 & = \Exx (p-q+1)! \, \mu_2^{p-q} \left(\mu_2 \sigma' + |V'^T V'|\right) \, (p-q-1)! \, \mu_2^{p-q-1} |\stkout{V}'^T \stkout{V}'|\\
    & = (p-q+1)!(p-q-1)! \, \mu_2^{2p-2q-1} \left(\mu_2\epsilon(p,q)+\eta(p,q)\right),
\end{split}
\end{equation}
inserting this result into \eqref{thet_np_pre}, we get the desired relation for $\epsilon(n,p)$.
\end{proof}

\begin{remark}
Dependencies of auxiliary moments on themselves are shown graphically in Figure \ref{DepMom}.
\end{remark}

\begin{figure}[htb!]
\centering
\begin{tikzpicture}[thick, main/.style = {draw,circle}, scale = 1.0]
\node[main] (f4) at (0.5,1.3) {$f_4$};
\node[main] (al) at (3.6,0.8) {$\alpha$};
\node[main] (be) at (1.1,0) {$\beta$}; 
\node[main] (ga) at (2.1,1.3) {$\gamma$};
\node[main] (de) at (2.5,0) {$\delta$}; 
\node[main] (ep) at (5,0) {$\epsilon$};
\node[main] (et) at (3.6,-0.8) {$\eta$};
\node[main] (ka) at (-0.5,0) {$\kappa$};
\node[main] (rh) at (2.1,-1.3) {$\rho$};
\node[main] (g4) at (0.5,-1.3) {$g_4$};
\draw[->] (al) to [bend right=60] (f4);
\draw[->] (be) to (f4);
\draw[->] (g4) to (be);
\draw[->] (g4) to [bend right=60] (et);
\draw[->] (f4) to [out=135,in=225,looseness=7] (f4);
\draw[->] (g4) to [out=135,in=225,looseness=7] (g4);
\draw[->] (ep) to [out=-45,in=45,looseness=7] (ep);
\draw[->] (de) to [out=135,in=225,looseness=7] (de);
\draw[->] (al) to [bend right=30] (ga);
\draw[->] (ga) to [bend right=30] (al);
\draw[->] (be) to [bend right=-30] (ga);
\draw[->] (rh) to [bend right=-30] (be);
\draw[->] (et) to [bend right=30] (rh);
\draw[->] (rh) to [bend right=30] (et);
\draw[->] (ka) to [bend right=30] (be);
\draw[->] (be) to [bend right=30] (ka);
\draw[->] (et) to [bend right=30] (ep);
\draw[->] (ep) to [bend right=30] (al);
\draw[->] (et) to [bend right=30] (al);
\draw[->] (de) to [bend right=15] (al);
\draw[->] (et) to [bend right=15] (de);
\end{tikzpicture}
\caption{Graph of dependencies in recurrence relations} \label{DepMom}
\end{figure}
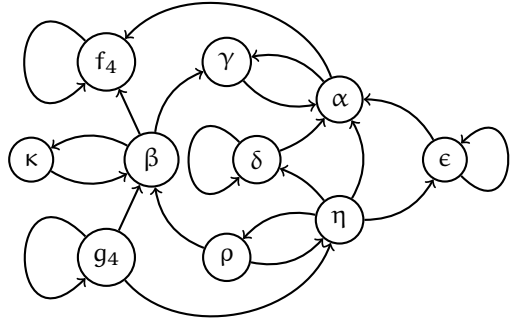

\subsection{Auxiliary generating functions}
\begin{definition}[$\hat{\alpha}(t,\omega),\hat{\beta}(t,\omega),\hat{\gamma}(t,\omega),\hat{\delta}(t,\omega),\hat{\epsilon}(t,\omega),\hat{\rho}(t,\omega),\hat{\eta}(t,\omega),\hat{\kappa}(t,\omega)$]\label{DefAuxMom}
We define the following generating functions
\begin{align*}
    \hat{\alpha}(t, \omega) & = \sum_{n=0}^\infty \sum_{p=0}^n\frac{(n-p)! }{n!p!} t^p \omega^{n-p} \alpha(n,p), \qquad \qquad \qquad\! \hat{\beta}(t, \omega) = \sum_{n=0}^\infty \sum_{p=0}^n\frac{(n-p)! }{n!p!} t^p \omega^{n-p} \beta(n,p), \\    \hat{\gamma}(t, \omega) & = \sum_{n=0}^\infty \sum_{p=0}^n\frac{(n-p+1)! }{n!p!} t^p \omega^{n-p+1} \gamma(n,p), \qquad \quad \, \hat{\delta}(t, \omega) = \sum_{n=0}^\infty \sum_{p=0}^n\frac{(n-p+1)! }{n!p!} t^p \omega^{n-p+1} \delta(n,p), \\
    \hat{\epsilon}(t, \omega) & = \sum_{n=0}^\infty \sum_{p=0}^n\frac{(n-p+1)! }{n!p!} t^p \omega^{n-p+1} \epsilon(n,p), \qquad \quad\, \hat{\eta}(t, \omega) = \sum_{n=0}^\infty \sum_{p=0}^n\frac{(n-p+1)! }{n!p!} t^p \omega^{n-p+1} \eta(n,p), \\
    \hat{\rho}(t, \omega) & = \sum_{n=0}^\infty \sum_{p=0}^n\frac{(n-p+2)! }{n!p!} t^p \omega^{n-p+2} \rho(n,p), \qquad \quad\, \hat{\kappa}(t, \omega) = \sum_{n=0}^\infty \sum_{p=0}^{n+1}\frac{(n-p+2)! }{n!p!} t^p \omega^{n-p+2} \kappa(n,p). \\
\end{align*}
\end{definition}
\begin{definition}[$\alpha_j(t),\beta_j(t),\gamma_j(t),\delta_j(t),\epsilon_j(t),\rho_j(t),\eta_j(t),\kappa_j(t)$]\label{DefAnsaAux} Also, we define the ansatz coefficients via expansions
\begin{align*}
    \hat{\alpha}(t, \omega) & = \sum_{j=0}^\infty \omega^j \left(\frac{1-\mu_2^2 t}{1-\omega-\mu_2^2 t}\right)^{j+1} \alpha_j(t), \qquad \qquad \qquad \! \hat{\beta}(t, \omega) = \sum_{j=0}^\infty \omega^j \left(\frac{1-\mu_2^2 t}{1-\omega-\mu_2^2 t}\right)^{j+1} \beta_j(t), \\    \hat{\gamma}(t, \omega) & = \sum_{j=0}^\infty \omega^j \left(\frac{1-\mu_2^2 t}{1-\omega-\mu_2^2 t}\right)^{j+1} \gamma_j(t), \qquad \qquad \qquad \hat{\delta}(t, \omega) = \sum_{j=0}^\infty \omega^j \left(\frac{1-\mu_2^2 t}{1-\omega-\mu_2^2 t}\right)^{j+1} \delta_j(t), \\
    \hat{\epsilon}(t, \omega) & = \sum_{j=0}^\infty \omega^j \left(\frac{1-\mu_2^2 t}{1-\omega-\mu_2^2 t}\right)^{j+1} \epsilon_j(t), \qquad \qquad \qquad \hat{\eta}(t, \omega) = \sum_{j=0}^\infty \omega^j \left(\frac{1-\mu_2^2 t}{1-\omega-\mu_2^2 t}\right)^{j+1} \eta_j(t), \\
    \hat{\rho}(t, \omega) & = \sum_{j=0}^\infty \omega^j \left(\frac{1-\mu_2^2 t}{1-\omega-\mu_2^2 t}\right)^{j+1} \rho_j(t), \qquad \qquad \qquad \hat{\kappa}(t, \omega) = \sum_{j=0}^\infty \omega^j \left(\frac{1-\mu_2^2 t}{1-\omega-\mu_2^2 t}\right)^{j+1} \kappa_j(t). \\
\end{align*}
\end{definition}

\begin{definition}[$G_4^*,\Psi_4^*$] On top of that, we define the following extra generating function
\begin{equation*}
    G_4^*(t, \omega) = \sum_{n=0}^\infty \sum_{p=0}^n\frac{(n-p+2)! }{n!p!} t^p \omega^{n-p+2} g_4(n,p) = \sum_{j=0}^\infty \omega^j \left(\frac{1-m_2^2 t}{1-\omega-m_2^2 t}\right)^{j+1} \Psi^*_j(t).
\end{equation*}
\end{definition}
\begin{proposition}
\begin{equation*}
    G_4^*(t,\omega) = 2\omega^2 \left(\frac{1-m_2^2 t}{1-\omega-m_2^2 t}\right)^3G_4(t),
\end{equation*}
that is $\Psi^*_2(t)=2G_4(t)$ and $\Psi^*_j(t)=0$ otherwise.
\end{proposition}
\begin{proof}
Note that
\begin{equation}
    G^*_4(t,\omega) = \omega \frac{\partial}{\partial\omega}\left(\omega \frac{\partial}{\partial\omega}\left(\omega G_4(t,\omega)\right)\right) = \sum_{j=0}^\infty \omega^{j+2} \left(\frac{1-m_2^2 t}{1-\omega-m_2^2 t}\right)^{j+3} (j+2)(j+1)\Psi_j(t),
\end{equation}
so in other words
\begin{equation}
    \Psi^*_j(t) = j(j-1)\Psi_{j-2}(t).
\end{equation}
Since $\Psi_0(t) = G_4(t)$ and otherwise $\Psi_j(t)$ is zero for $j\geq 1$, this finishes the proof.
\end{proof}

\begin{proposition}
\begin{align*}
    F_4(t,\omega) & = \frac{1}{1-\omega}\left( F_4\left(t,\frac{\omega \mu_2 ^2 t}{1-\omega}\right)+\frac{2m_1^2}{\mu_2}\hat{\alpha}\left(t,\frac{\omega \mu_2 ^2 t}{1-\omega}\right)+\frac{m_1^4}{\mu_2^2}\hat{\beta}\left(t,\frac{\omega \mu_2 ^2 t}{1-\omega}\right)\right), \\
    \hat{\alpha}(t,\omega) & = \frac{1}{\mu_2^2t(1-\omega)}\left(\mu_2\hat{\gamma}\left(t,\frac{\omega \mu_2^2 t}{1-\omega}\right)+\hat{\delta}\left(t,\frac{\omega \mu_2 ^2 t}{1-\omega}\right)+m_1^2\hat{\epsilon}\left(t,\frac{\omega \mu_2 ^2 t}{1-\omega}\right)+\frac{m_1^2}{\mu_2}\hat{\eta}\left(t,\frac{\omega \mu_2 ^2 t}{1-\omega}\right)\right),\\
    \hat{\beta}(t,\omega) & = \frac{1}{\mu_2^4t^2(1-\omega)}\left(\mu_2^2\hat{\kappa}\left(t,\frac{\omega \mu_2 ^2 t}{1-\omega}\right) + 2\mu_2\hat{\rho}\left(t,\frac{\omega \mu_2 ^2 t}{1-\omega}\right)+G_4^*\left(t,\frac{\omega \mu_2 ^2 t}{1-\omega}\right)\right),\\
    \hat{\gamma}(t,\omega) & = \frac{1}{\mu_2 (1-\omega)} \left(\hat{\alpha}\left(t,\frac{\omega \mu_2 ^2 t}{1-\omega}\right)+\frac{m_1^2}{\mu_2}\hat{\beta}\left(t,\frac{\omega \mu_2 ^2 t}{1-\omega}\right)\right),\\
    \hat{\delta}(t,\omega) & = \frac{1}{\mu_2^2t(1-\omega)} \left(\hat{\delta}\left(t,\frac{\omega \mu_2 ^2 t}{1-\omega}\right)+\frac{m_1^2}{\mu_2}\hat{\eta}\left(t,\frac{\omega \mu_2 ^2 t}{1-\omega}\right)\right),\\
    \hat{\epsilon}(t,\omega) & = \frac{1}{\mu_2^3t(1-\omega)}\left(\mu_2\hat{\epsilon}\left(t,\frac{\omega \mu_2 ^2 t}{1-\omega}\right)+\hat{\eta}\left(t,\frac{\omega \mu_2 ^2 t}{1-\omega}\right)\right),\\
    \hat{\eta}(t,\omega) & = \frac{1}{\mu_2^4t^2(1-\omega)}\left(\mu_2\hat{\rho}\left(t,\frac{\omega \mu_2 ^2 t}{1-\omega}\right)+G^*_4\left(t,\frac{\omega \mu_2 ^2 t}{1-\omega}\right)\right),\\
    \hat{\rho}(t,\omega) & = \frac{1}{\mu_2^3t(1-\omega)} \hat{\eta}\left(t,\frac{\omega \mu_2 ^2 t}{1-\omega}\right),\\
    \hat{\kappa}(t,\omega) & = \frac{1}{\mu_2^2(1-\omega)} \hat{\beta}\left(t,\frac{\omega \mu_2 ^2 t}{1-\omega}\right).
    \end{align*}
\end{proposition}
\begin{proof}
Insert Propositions \ref{PropSeqf4} -- \ref{PropSeqEpsi} into the definitions of the corresponding generating functions.
\end{proof}
\begin{corollary}
In the terms of ansatz coefficients, comparing the terms of the expansions, this is equal to the linear system
\begin{align*}
    \Phi_j(t) & = \mu_2^{2j} t^{j} \left(\Phi_j(t) +\frac{2m_1^2}{\mu_2}\alpha_j(t)+\frac{m_1^4}{\mu_2^2}\beta_j(t)\right),
    && \alpha_j(t) =\mu_2^{2j-2}t^{j-1}\left(\mu_2\gamma_j(t)+\delta_j(t)+m_1^2\epsilon_j(t)+\frac{m_1^2}{\mu_2}\eta_j(t)\right),\\
    \beta_j(t) & = \mu_2^{2j-4}t^{j-2}\left(\mu_2^2\kappa_j(t)+2\mu_2\rho_j(t)+\Psi^*_j(t)\right),
    && \gamma_j(t) = \mu_2^{2j-1}t^{j}\left(\alpha_j(t)+\frac{m_1^2}{\mu_2}\beta_j(t)\right),\\
    \delta_j(t) & = \mu_2^{2j-2}t^{j-1}\left(\delta_j(t)+\frac{m_1^2}{\mu_2}\eta_j(t)\right),
    && \epsilon_j(t) = \mu_2^{2j-3}t^{j-1}\left(\mu_2\epsilon_j(t)+\eta_j(t)\right),\\
    \eta_j(t) & = \mu_2^{2j-4}t^{j-2}\left(\mu_2\rho_j(t)+\Psi^*_j(t) \right),\qquad
    && \rho_j(t) = \mu_2^{2j-3}t^{j-1}\eta_j(t), \qquad \qquad \qquad \,\,\,
    \kappa_j(t) = \mu_2^{2j-2}t^j\beta_j(t).
    \end{align*}
\end{corollary}

\begin{corollary}\label{CoroResolLin}
Solving the linear system, we get
\begin{align*}
& && \Phi_1(t) = m_1^2t\,\frac{2\mu_2\alpha_1(t)+m_1^2\beta_1(t)}{1-\mu_2^2t}, && \Phi_2(t) = \frac{2m_1^4\mu_2^2 t^2}{(1-\mu_2^2t)^4}G_4(t),\\
\alpha_0(t) & = 0, && \alpha_1(t) = \frac{\delta_1(t)+m_1\epsilon_1(t)+m_1^2\mu_2 t \beta_1(t)}{1-\mu_2^2t}, && \alpha_2(t) = \frac{2m_1^2\mu_2 t}{(1-\mu_2^2t)^3}G_4(t),\\
\beta_0(t) & = 0, && && \beta_2(t) = \frac{2}{(1-\mu_2^2t)^2}G_4(t),\\
\gamma_0(t) & = 0, && \gamma_1(t) = t\,\frac{\mu_2\delta_1(t)+m_1\mu_2\epsilon_1(t)+m_1^2 \beta_1(t)}{1-\mu_2^2t}, && \gamma_2(t) = \frac{2m_1^2\mu_2^2 t^2}{(1-\mu_2^2t)^3}G_4(t),\\
\delta_0(t) & = 0, && && \delta_2(t) = \frac{2m_1^2\mu_2 t}{(1-\mu_2^2t)^2}G_4(t),\\
\epsilon_0(t) & = 0, && && \epsilon_2(t) = \frac{2\mu_2 t}{(1-\mu_2^2t)^2}G_4(t),\\
\eta_0(t) & = 0, && \eta_1(t) = 0, && \eta_2(t) = \frac{2}{1-\mu_2^2t}G_4(t),\\
\rho_0(t) & = 0, && \rho_1(t) = 0, && \rho_2(t) = \frac{2\mu_2t}{1-\mu_2^2t}G_4(t),\\
\kappa_0(t) & = 0, && \kappa_1(t) = t\beta_1(t), && \kappa_2(t) = \frac{2\mu_2^2t^2}{(1-\mu_2^2t)^2}G_4(t)
\end{align*}
and
\begin{equation*}
\Phi_j(t) = \alpha_j(t) = \beta_j(t) = \gamma_j(t) = \delta_j(t) = \epsilon_j(t) = \eta_j(t) = \rho_j(t) = \kappa_j(t) = 0 \quad \text{for} \quad j\geq 3.
\end{equation*}
\end{corollary}

\subsection{The remaining ansatz coefficients and the final conclusion}
\begin{proposition}\label{PropAlBe}
\begin{align*}
    \epsilon_1(t) & = \sum_{n=0}^\infty \frac{\epsilon(n,n)}{n!^2} t^n, && \delta_1(t) = \sum_{n=0}^\infty \frac{\delta(n,n)}{n!^2} t^n, && \beta_1(t) = \sum_{n=0}^\infty \frac{\beta(n+1,n)}{(n+1)!n!} t^n.
\end{align*}
\end{proposition}
\begin{proof}
In Definition \ref{DefAnsaAux}, perform the Taylor expansion in $\omega$ and then compare with Definition \ref{DefAuxMom}.
\end{proof}

\begin{proposition}\label{EpsProp}
\begin{equation*}
\epsilon_1(t) = 0
\end{equation*}
\end{proposition}
\begin{proof}
Trivial since $\epsilon(n,n)=0$ (see Remark \ref{RemSpecVals}).
\end{proof}

\begin{proposition}\label{DelProp}
\begin{equation*}
\delta_1(t) = \frac{1+\left(m_1^2 \mu_2-\mu_2^2+2 m_1 \mu_3\right) t+ \left(m_1^2 \mu_3^2-2 m_1 \mu_2^2 \mu_3\right) t^2-m_1^2 \mu_2^2 \mu_3^2 t^3}{1-\mu_2^2 t} G_4(t).
\end{equation*}
\end{proposition}
\begin{proof}
With $U$ and $V$ having dimensions $n\times n$, we have similarly as in the previous proposition,
\begin{equation}
    \delta(n,n) = \Exx |U^TU| |V^T V| = \Exx |A|^2 |B|^2,
\end{equation}
which can be now simplified using standard techniques as before. Recall the formula
\begin{equation}
    |A| = |B| + m_1 S, \qquad\text{where}\qquad S = \sum_{ij} (-1)^{i+j}|B_{ij}|,
\end{equation}
then
\begin{equation}\label{Delta0}
    \delta(n,n) = \Exx |B|^4  +  2m_1 \Exx |B|^3 S  +  m_1^2 \Exx |B|^2 S^2.
\end{equation}
It turns out that the terms are either trivial or already expressed, namely
\begin{itemize}
    \item $\Exx |B|^4 = g_4(n)$
    \item $\Exx |B|^3 S = n^2 \mu_3 g_4(n-1)$
    \item $\Exx |B|^2 S^2 = n^2 h_0(n)+n^2(n-1)^2 \mu_3^2 g_4(n-2)$
\end{itemize}
In total, multiplying \eqref{Delta0} by $t^n/n!^2$ and summing up,
\begin{equation}
    \delta_0(t) = G_4(t) + 2m_1 \mu_3 t  G_4(t) + m_1^2 \left( H_0(t) + \mu_3^2 t^2 G_4(t)\right),
\end{equation}
from which the proposition follows since $H_0(t) = \mu_2 t G_4(t)/(1-\mu_2^2 t)$ (see Proposition \ref{PropH}).
\end{proof}

\begin{proposition}\label{BetProp}
\begin{equation*}
\beta_1(t) = \frac{1+2\mu_2^2 t}{1-\mu_2^2 t} G_4(t)
\end{equation*}
\end{proposition}
\begin{proof}
Let $V$ has dimensions $(n+1) \times n$, thus $\stkout{V}$ is a square $(n+1) \times (n+1)$ matrix. By the multiplicative property of determinant,
\begin{equation}
    \beta(n+1,n) = \Exx |\stkout{V}^T \stkout{V}|^2 = \Exx |\stkout{V}|^4.
\end{equation}
By expansion in the last column,
\begin{equation}
|\stkout{V}| = \sum_{i=1}^{n+1} (-1)^{i+n+1}|V_{[i]}|.
\end{equation}
Hence
\begin{equation}
    \beta(n+1,n) = \Exx |\stkout{V}|^4 = \Exx \sum_{i,j,r,s \in \{1,\ldots, n+1 \}} (-1)^{i+j+r+s}|V_{[i]}||V_{[j]}||V_{[r]}||V_{[s]}|.
\end{equation}
By symmetry in $ijrs$ and by omitting trivially vanishing terms, we obtain
\begin{equation}
    \beta(n+1,n)=(n+1)\Exx |V_{[1]}|^4 + 3(n+1)n \Exx |V_{[1]}|^2 |V_{[2]}|^2,
\end{equation}
which we can write as, recalling $h_8(n)$ and $B$ the $n \times n$ matrix of $Y_{ij}$'s,
\begin{equation}
    \frac{\beta(n+1,n)}{n+1} = \Exx |B|^ 4 + 3n \left[\Exx |B_{11}|^2|B_{12}|^2\right]_{n\rightarrow n+1} = g_4(n) + 3 n^2 \mu_2 h_0(n).
\end{equation}
In total, summing for all $n$ and by using the definitions of the generating functions,
\begin{equation}
    \beta(t) = G_4(t) + 3 \mu_2 H_0(t) = \frac{1+2\mu_2^2 t}{1-\mu_2^2 t} G_4(t).
\end{equation}
\end{proof}

\begin{proposition}\label{AlpProp}
Combining Propositions \ref{EpsProp} -- \ref{BetProp} and by Corollary \ref{CoroResolLin}, we get
\begin{equation*}
\alpha_1(t) = \frac{1+\left(2 m_1^2 \mu _2-\mu _2^2+2 m_1 \mu _3\right) t+\left(2 m_1^2 \mu _2^3-2 m_1 \mu _2^2 \mu _3+m_1^2 \mu _3^2\right) t^2-m_1^2 \mu _2^2 \mu _3^2 t^3}{\left(1-\mu_2^2 t\right)\! {}^2} G_4(t).
\end{equation*}
\end{proposition}

\begin{proposition}
Combining Propositions \ref{BetProp} and \ref{AlpProp} and by Corollary \ref{CoroResolLin}, we get, defining $\tilde{p}_k$'s as before,
\begin{equation*}
    \Phi_1(t) = \frac{m_1^2 G_4(t)}{(1 - \mu_2^2 t)^3} \sum_{k=1}^4 \tilde{p}_k t^k.
\end{equation*}
\end{proposition}

\begin{corollary}
Together with the fact that $\Phi_0(t) = F_4(t)$ from \eqref{SpecPhi0}, we recover Theorem \ref{MainThmDembo}.
\end{corollary}

\section{Final remarks}
We believe it might be still possible to derive the full $f_6(n)$ via the same treatment as presented in this paper (Lemma \ref{MainLem} and expansions in all classes). Similarly, by cubing Cauchy-Binet formula, one could obtain $f^{\mathrm{sym}}_6(n,p)$ and possibly $f_6(n,p)$. But that task may be way harder.


\printbibliography[heading=bibintoc]


\appendix
\section{Matrix symbols}

\vfill

\renewcommand{\tupsingle}{\ru{3}} 
\renewcommand{\arraystretch}{1.5}
   \begin{table}[h]
\centering
\setlength{\tabcolsep}{2pt} 
\setlength\jot{1pt} 
\begin{tabular}{|c|c|c|c|c|c|c|c|}
 \hline

\tupsingle
\begin{tabular}{c} $\mBudef{\mB}$ \\ $B$ \end{tabular} &
\begin{tabular}{c} $\mBudef{\mBaj}$ \\ $B/1j$ \end{tabular} &
\begin{tabular}{c} $\mBudef{\mBia}$ \\ $B/i1$ \end{tabular} &
\begin{tabular}{c} $\mBudef{\mBib}$ \\ $B/i2$ \end{tabular} &
\begin{tabular}{c} $\mBudef{\mBaa}$ \\ $B_{11}$ \end{tabular} &
\begin{tabular}{c} $\mBudef{\mBaabj}$ \\ $B_{11}/2j$ \end{tabular} &
\begin{tabular}{c} $\mBudef{\mBaacj}$ \\ $B_{11}/3j$ \end{tabular} &
\begin{tabular}{c} $\mBudef{\mBaadj}$ \\ $B_{11}/4j$ \end{tabular} \\
 \hline

\tupsingle
\begin{tabular}{c} $\mBudef{\mBaaib}$ \\ $B_{11}/i2$ \end{tabular} &
\begin{tabular}{c} $\mBudef{\mBaaic}$ \\ $B_{11}/i3$ \end{tabular} &
\begin{tabular}{c} $\mBudef{\mBaaid}$ \\ $B_{11}/i4$ \end{tabular} &
\begin{tabular}{c} $\mBudef{\mBaaie}$ \\ $B_{11}/i5$ \end{tabular} &
\begin{tabular}{c} $\mBudef{\mBab}$ \\ $B_{12}$ \end{tabular} &
\begin{tabular}{c} $\mBudef{\mBabbj}$ \\ $B_{12}/2j$ \end{tabular} &
\begin{tabular}{c} $\mBudef{\mBabcj}$ \\ $B_{12}/3j$ \end{tabular} & 
\begin{tabular}{c} $\mBudef{\mBabia}$ \\ $B_{12}/i1$ \end{tabular} \\
 \hline

\tupsingle
\begin{tabular}{c} $\mBudef{\mBabic}$ \\ $B_{12}/i3$ \end{tabular} &
\begin{tabular}{c} $\mBudef{\mBac}$ \\ $B_{13}$ \end{tabular} &
\begin{tabular}{c} $\mBudef{\mBacbj}$ \\ $B_{13}/2j$ \end{tabular} &
\begin{tabular}{c} $\mBudef{\mBacia}$ \\ $B_{13}/i1$ \end{tabular} &
\begin{tabular}{c} $\mBudef{\mBacib}$ \\ $B_{13}/i2$ \end{tabular} &
\begin{tabular}{c} $\mBudef{\mBad}$ \\ $B_{14}$ \end{tabular} & 
\begin{tabular}{c} $\mBudef{\mBadia}$ \\ $B_{14}/i1$ \end{tabular} &
\begin{tabular}{c} $\mBudef{\mBba}$ \\ $B_{21}$ \end{tabular} \\
\hline

\tupsingle
\begin{tabular}{c} $\mBudef{\mBbaaj}$ \\ $B_{21}/1j$ \end{tabular} &
\begin{tabular}{c} $\mBudef{\mBbaib}$ \\ $B_{21}/i2$ \end{tabular} &
\begin{tabular}{c} $\mBudef{\mBbaid}$ \\ $B_{21}/i4$ \end{tabular} &
\begin{tabular}{c} $\mBudef{\mBbb}$ \\ $B_{22}$ \end{tabular} &
\begin{tabular}{c} $\mBudef{\mBbbaj}$ \\ $B_{22}/1j$ \end{tabular} &
\begin{tabular}{c} $\mBudef{\mBbbcj}$ \\ $B_{22}/3j$ \end{tabular} &
\begin{tabular}{c} $\mBudef{\mBbbia}$ \\ $B_{22}/i1$ \end{tabular} &
\begin{tabular}{c} $\mBudef{\mBbbid}$ \\ $B_{22}/i4$ \end{tabular}\\
 \hline

\tupsingle
\begin{tabular}{c} $\mBudef{\mBbcia}$ \\ $B_{23}/i1$ \end{tabular} &
\begin{tabular}{c} $\mBudef{\mBbcib}$ \\ $B_{23}/i2$ \end{tabular} &
\begin{tabular}{c} $\mBudef{\mBbdib}$ \\ $B_{24}/i2$ \end{tabular} &
\begin{tabular}{c} $\mBudef{\mBca}$ \\ $B_{31}$ \end{tabular} &
\begin{tabular}{c} $\mBudef{\mBcabj}$ \\ $B_{31}/2j$ \end{tabular} &
\begin{tabular}{c} $\mBudef{\mBcb}$ \\ $B_{32}$ \end{tabular} &
\begin{tabular}{c} $\mBudef{\mBccaj}$ \\ $B_{33}/1j$ \end{tabular} &
\begin{tabular}{c} $\mBudef{\mBccbj}$ \\ $B_{33}/2j$ \end{tabular}\\
 \hline

\tupsingle
\begin{tabular}{c} $\mBudef{\mBcdia}$ \\ $B_{34}/i1$ \end{tabular} &
\begin{tabular}{c} $\mBudef{\mBddaj}$ \\ $B_{44}/1j$ \end{tabular} &
\begin{tabular}{c} $\mBudef{\mBabab}$ \\ $B_{12,12}$ \end{tabular} &
\begin{tabular}{c} $\mBudef{\mBabac}$ \\ $B_{12,13}$ \end{tabular} &
\begin{tabular}{c} $\mBudef{\mBabad}$ \\ $B_{12,14}$ \end{tabular} &
\begin{tabular}{c} $\mBudef{\mBabbc}$ \\ $B_{12,23}$ \end{tabular} &
\begin{tabular}{c} $\mBudef{\mBabbd}$ \\ $B_{12,24}$ \end{tabular} &
\begin{tabular}{c} $\mBudef{\mBabbe}$ \\ $B_{12,25}$ \end{tabular}\\
 \hline

\tupsingle
\begin{tabular}{c} $\mBudef{\mBabcd}$ \\ $B_{12,34}$ \end{tabular} &
\begin{tabular}{c} $\mBudef{\mBacab}$ \\ $B_{13,12}$ \end{tabular} &
\begin{tabular}{c} $\mBudef{\mBacac}$ \\ $B_{13,13}$ \end{tabular} &
\begin{tabular}{c} $\mBudef{\mBacbc}$ \\ $B_{13,23}$ \end{tabular} &
\begin{tabular}{c} $\mBudef{\mBaccd}$ \\ $B_{13,34}$ \end{tabular} &
\begin{tabular}{c} $\mBudef{\mBacce}$ \\ $B_{13,35}$ \end{tabular} &
\begin{tabular}{c} $\mBudef{\mBadab}$ \\ $B_{14,12}$ \end{tabular} &
\begin{tabular}{c} $\mBudef{\mBadad}$ \\ $B_{14,14}$ \end{tabular}\\
 \hline

\tupsingle
\begin{tabular}{c} $\mBudef{\mBadde}$ \\ $B_{14,45}$ \end{tabular} & 
\begin{tabular}{c} $\mBudef{\mBbcab}$ \\ $B_{23,12}$ \end{tabular} &
\begin{tabular}{c} $\mBudef{\mBbcad}$ \\ $B_{23,14}$ \end{tabular} &
\begin{tabular}{c} $\mBudef{\mBbcbc}$ \\ $B_{23,23}$ \end{tabular} &
\begin{tabular}{c} $\mBudef{\mBbcbd}$ \\ $B_{23,24}$ \end{tabular} &
\begin{tabular}{c} $\mBudef{\mBbccd}$ \\ $B_{23,34}$ \end{tabular} &
\begin{tabular}{c} $\mBudef{\mBbdac}$ \\ $B_{24,13}$ \end{tabular} &
\begin{tabular}{c} $\mBudef{\mBcdad}$ \\ $B_{34,14}$ \end{tabular} \\
 \hline

    \end{tabular} 
\caption{Table of all matrix symbols used}
\label{AllmB}
\end{table}

\vfill
\begin{equation*}
    \Exx |B_{11}|^2 |B_{23,24}|^2 |B_{22}| |B_{21}| = \Ex \left(\mBaaid\right)^2 \left(\mBbcbd  \right)^2 \mBbbid \,\, \mBbaid
\end{equation*}

\end{document}